\documentclass[a4paper,11pt]{article}

\usepackage{amssymb}
\usepackage{graphicx}
\usepackage{tikz}
\usetikzlibrary{intersections}
\usetikzlibrary{svg.path}
\usetikzlibrary{decorations.pathmorphing}

\usepackage{pgfplots}
\pgfplotsset{compat=1.11}
\usepackage{mathrsfs}
\usepackage{amsmath}
\usepackage{amsthm}
\usepackage{enumerate}
\usepackage{color}
\usepackage{verbatim}
\usepackage{todonotes}

\usepackage[margin=1in]{geometry}
\usepackage{array}
\newcolumntype{e}{>{\displaystyle}r @{\,} >{\displaystyle}c @{\,} >{\displaystyle}l}

\newcommand{\mcup}{\textstyle \bigcup\limits}

\usepackage{amssymb,amsopn,amscd}

\usepackage[colorlinks]{hyperref}
\hypersetup{final}
\usepackage{xcolor}
\hypersetup{
    colorlinks,
    linkcolor={blue!40!black},
    citecolor={blue!40!black},
    urlcolor={blue!80!black}
}
\usepackage{nameref}

\usepackage[figure,table]{hypcap} 

\theoremstyle{plain}
\newtheorem{Theorem}{Theorem}[section]
\newtheorem{Corollary}[Theorem]{Corollary}
\newtheorem{Lemma}[Theorem]{Lemma}
\newtheorem{Proposition}[Theorem]{Proposition}

\newtheorem{Question}[Theorem]{Question}
\newtheorem{Claim}[Theorem]{Claim}

\theoremstyle{definition}

\newtheorem{Definition}[Theorem]{Definition}
\newtheorem{Remark}[Theorem]{Remark}

\newcounter{Condition}
\newtheorem{Condition}[Condition]{Condition}

\renewcommand{\P}{\mathbb{P}}

\newcommand{\R}{\mathbb{R}}
\newcommand{\N}{\mathbb{N}}

\renewcommand{\S}{\mathcal{S}}

\newcommand{\V}{\mathcal{V}}

\DeclareMathOperator{\diam}{diam}
\DeclareMathOperator{\Range}{Range}

\numberwithin{equation}{section}

\newcommand{\crab}[3]{%
  \draw[color=white, very thick] (#1, #2) circle (1.3);
  \draw (#1, #2) circle (1.3);
  \draw[fill=gray] (#1 - .3, #2 - 1.6) circle (.05);
  \draw[fill=gray] (#1 - 1, #2 + .2) circle (.05);
  \draw[fill=gray] (#1 + .3, #2 - 1.5) circle (.05);
  \draw[fill=gray] (#1 + 1.4, #2 - 1) circle (.05);
  \node[below,right] at (#1, #2 + .1) {\scriptsize $x_{_{#3}}$};
  \draw[fill=black] (#1, #2) circle (.02);
  \draw[fill=gray] plot [smooth cycle] coordinates { (#1 - .3, #2 - .3) (#1 - .6, #2) (#1 - .7, #2 - .2) (#1 - .45, #2 - .4)
        (#1 - .2, #2 - .6) (#1 - .1, #2 - .4) };
  \node[below,left] at (#1 - .55, #2 - .3) {\scriptsize $A^{^0}_{_{#3}}$};
  \draw[fill=gray] plot [smooth cycle] coordinates { (#1 + .4, #2 - .2) (#1 + .7, #2) (#1 + .7, #2 - .3) (#1 + .6, #2 - .4)
        (#1 + .2, #2 - .7) (#1 + .3, #2 - .4) };
  \node[below,right] at (#1 + .45, #2 - .5) {\scriptsize $A^{^1}_{_{#3}}$};
}

\def\htrone{.4}
\def\htrtwo{.4}

  \newcounter{constant}
  \newcommand{\nc}[1]{\refstepcounter{constant}\label{#1}}
  \newcommand{\uc}[1]{c_{\textnormal{\tiny \ref{#1}}}}
\usepackage{calc}
\def\arraypar#1{\parbox[c]{\textwidth - 2cm}{\centering #1}}

\usepackage{environ}
\NewEnviron{display}{
\begin{equation}\begin{array}{c}
  \arraypar{\BODY}
\end{array}\end{equation}
}

\begin{document}

\title{Percolation and isoperimetry on roughly transitive graphs}

\author{Elisabetta Candellero\footnote{Email: \ E.Candellero@warwick.ac.uk; \ University of Warwick, Dept of Statistics, CV4 7AL Coventry, UK} \and Augusto Teixeira\footnote{Email: \ augusto@impa.br; \ IMPA, Estrada Dona Castorina 110, 22460-320, Rio de Janeiro, RJ, Brazil}}

\maketitle

\begin{abstract}
In this paper we study percolation on a roughly transitive graph $G$ with polynomial growth and isoperimetric dimension larger than one.
For these graphs we are able to prove that $p_c < 1$, or in other words, that there exists a percolation phase.
The main results of the article work for both dependent and independent percolation processes, since they are based on a quite robust renormalization technique.
When $G$ is transitive, the fact that $p_c < 1$ was already known before.
But even in that case our proof yields some new results and it is entirely probabilistic, not involving the use of Gromov's theorem on groups of polynomial growth.
We finish the paper giving some examples of dependent percolation for which our results apply.
\end{abstract}


\section{Introduction}

Since its introduction by Broadbent and Hammersley in \cite{PSP:2048852}, the model of independent percolation has received major attention from the physical and mathematical communities.
From the perspective of applications, it has the potential to model several different systems, from the flow of fluids in porous media, to the transmission of information on networks or diseases on populations.
On the theoretical side, this model has been source of challenging questions, and has given rise to beautiful theories.
For a mathematical background of the model on $\mathbb{Z}^d$, see \cite{Gri99} and \cite{bollobas2006percolation} and the references therein.

Besides the classical independent model on $\mathbb{Z}^d$, this study has been generalized by both considering the model on more general graphs, see for instance \cite{LP11}, \cite{BS96}, \cite{HJ06} and \cite{zbMATH05636419}, or by adding dependence to the percolation configuration, see \cite{zbMATH01018381}, \cite{zbMATH00846204}, \cite{Szn09} and \cite{TW10b} for some examples of such works.

In this article, we study vertex percolation on roughly transitive graphs, with or without dependence, showing the existence of a phase transition for the process as we vary the density of open vertices.
Another important contribution of this work is to help develop multi-scale renormalization on roughly transitive graphs of polynomial growth.
Renormalization is a powerful tool, which has been used to analyze several stochastic processes.
However this technique has limitations that often restrict its use to the lattice $\mathbb{Z}^d$.

\subsection{Graphs under consideration}

\nc{c:rough_trans}
In this paper we consider both dependent and independent percolation on roughly transitive graphs.
To define this concept precisely, we need to first introduce the notion of rough isometries.

Given graphs $G$, $G'$ and a constant $\uc{c:rough_trans} \geq 1$, a map $\phi:G \to G'$ is said to be a $\uc{c:rough_trans}$-rough isometry if for any~$x, y \in G$ we have
\begin{equation}
  \label{e:rough_iso}
  \frac{1}{\uc{c:rough_trans}} \; d\big( x, y \big) - 1 < d(\phi(x), \phi(y)) \leq \uc{c:rough_trans} \; d\big( x, y \big)
\end{equation}
and for any $y \in G'$, there exists some $x \in G$ such that
\begin{equation}
  \label{e:rough_surj}
  d(\phi(x), y) \leq \uc{c:rough_trans}.
\end{equation}
We say that a given graph $G$ is $\uc{c:rough_trans}$-roughly transitive if for any $x, y \in G$ there exists a $\uc{c:rough_trans}$-rough isometry $\phi$ satisfying $\phi(x) = y$.

\begin{Remark}
There are other (equivalent) definitions of rough isometry, see e.g., Definition~3.7 of \cite{W00}.
In this work, it is convenient to use \eqref{e:rough_iso} together with \eqref{e:rough_surj} as used for example in \cite{elek}.
\end{Remark}

\bigskip

In \cite{BS96}, Benjamini and Schramm suggested a connection between the existence of a phase transition for independent percolation on a given graph and its isoperimetric dimension.

\begin{Definition}
  We say that $ G=(V,E)$ satisfies the isoperimetric inequality $\mathcal{I}(c_i, d_i)$ if
  \begin{equation}\label{eq:isoperimetric}
    \textnormal{for any finite set }A\subseteq V, \textnormal{ we have }|\partial A|\geq c_i |A|^{\frac{d_i-1}{d_i}},
  \end{equation}
for some suitable constant $c_i>0$ and a real number $d_i>1$.
\end{Definition}

For example, it is not difficult to see that $\mathbb{Z}^d$ satisfies the $\mathcal{I}(c, d)$ for some $c > 0$, see Theorem 6.37 of \cite{LP11}, p. 210.
In \cite[Question~2]{BS96}, Benjamini and Schramm asked this:
\begin{Question}\label{question:BS96}
  Is it true that if $G$ satisfies $\mathcal{I}(c_i, d_i)$ for some $d_i > 1$ then $p_c(G) < 1$?
\end{Question}
See the precise definition of $p_c(G)$ in \eqref{e:p_c} below.
In this article we give a positive answer to the above question in the case of roughly transitive graphs of polynomial growth.

Isoperimetric conditions and independent percolation have been studied in various works.
In \cite{BS96}, the authors proved that $p_c(G) < 1$ when $G$  has \emph{infinite isoperimetric dimension} (meaning that \eqref{eq:isoperimetric} holds with $(d_i - 1)/d_i$ replaced by $1$).
In \cite{zbMATH05229215}, Kozma showed that for edge percolation $p_c(G) < 1$ when $G$ is a \emph{planar} graph with isoperimetric dimension strictly larger than one, polynomial growth and no accumulation points.
In \cite{2014arXiv1409.5923T}, a stronger version of \eqref{eq:isoperimetric} called \emph{local isoperimetric inequality} was shown to imply $p_c(G) < 1$ for graphs with polynomial growth.
Some arguments in this paper are very similar in spirit to those of \cite{2014arXiv1409.5923T}, the main novelty being that we can replace the stronger \emph{local isoperimetric inequality} of \cite{2014arXiv1409.5923T} by the classical \eqref{eq:isoperimetric} in the case of roughly transitive graphs.

In this paper we deal with graphs with polynomial growth, as specified in the following.
\begin{Definition}
Given constants $c_u, d_u > 0$, we say that $G$ satisfies  $\V(c_u, d_u)$ if for every $r \geq 1$ and $x \in V$
\begin{equation}
  \label{eq:volume_upper_bound}
  |B(x,r)| \leq c_u r^{d_u}.
\end{equation}
We then say that $G$ has \emph{polynomial growth} if there are constants $c_u, d_u>0$ such that $G$ satisfies $\V(c_u, d_u)$.
\end{Definition}

\subsection{Main result}

The first result we present here is the following.

\begin{Theorem}
  \label{thm:p_c_Bernoulli}
  If $G$ is a roughly transitive graph of polynomial growth satisfying \eqref{eq:isoperimetric} with $c_i>0$ and $d_i>1$, then $p_c(G) < 1$.
  This gives a positive answer to Question~\ref{question:BS96} in this special case.
\end{Theorem}

\begin{Remark}
Let us note that whenever $p_c(G) < 1$ and $G$ has bounded degree, then the graph $G$ also undergoes a non-trivial phase transition for the Ising model, the Widom-Rowlinson model and the beach model. This follows from Theorems~1.1 and 1.2 of \cite{MR1765172}.
Moreover, it is a consequence of \cite[Remark 6.2]{lyons-schramm}, that if on a bounded-degree graph one has $p_c(V(G))<1$ for Bernoulli percolation on the \emph{vertices}, then one also has $p_c(E(G))<1$ for Bernoulli percolation on the edges.
\end{Remark}

The above result is a consequence of our Theorem~\ref{thm:p_c<1_dependent} below, which applies to both dependent and independent percolation processes.
Roughly speaking, Theorem~\ref{thm:p_c<1_dependent} states that, if the dependencies decay fast enough with the distance, then the percolation undergoes a non-trivial phase transition.
To be more precise, we need to define what we mean by ``decay of dependence''.

Let $\P$ denote any probability measure on the state space $\Omega:=\{0,1\}^V$, endowed with the $\sigma$-algebra generated by the canonical projections $Y_x:\Omega \to \{0,1\}$, defined by $Y_x(\omega):=\omega(x)$, for $x \in V$.

Whenever we say that the marginals of $\{Y_x\}$ are ``\emph{large} (resp.\ small) enough'', we mean that we require a large enough lower bound (resp.\ small enough upper bound) which is uniform over all $x\in V$.
Note that this can depend on the parameters that appear in the context, but not on the measure $\P$ itself.
\begin{Definition}\label{def:decoupling}
We shall say that $\P$ satisfies the \emph{decoupling inequality} $ \mathcal{D}(\alpha,c_\alpha)$ (where $\alpha>0$ is a fixed parameter) if for any $x \in V$, $r\geq 1$ and two events $ \mathcal{G}$ and $\mathcal{G}'$ such that
\[
\mathcal{G}\in \sigma\bigl ( Y_z, z\in B(x,r)\bigr ) \qquad \text{and}\qquad \mathcal{G}'\in \sigma\bigl ( Y_w,\,  d(w,x)\geq 2r\bigr ),
\]
we have
\[
\P(\mathcal{G} \cap \mathcal{G}')\leq \bigl (\P (\mathcal{G}) + c_\alpha r^{-\alpha}\bigr ) \P(\mathcal{G}').
\]
For convenience, we always assume that $B(x,r) $ is the set of elements lying at distance smaller than or equal to $r$ from $x$.
\end{Definition}

We are now in position to state the following.

\begin{Theorem}\label{thm:p_c<1_dependent}
Let $G$ be a $\uc{c:rough_trans}$-roughly transitive graph satisfying $\mathcal{V}(c_u, d_u)$ and $\mathcal{I}(c_i, d_i)$, with $d_i > 1$ and assume the law $\P$ satisfies $\mathcal{D}(\alpha, c_\alpha)$ with $\alpha > \alpha_{\ast}$ (see Remark~\ref{remark:alpha_star} for the definition of $\alpha_{\ast}$).
Then there exists a $p_\ast < 1$, depending only on $\alpha, c_\alpha,\uc{c:rough_trans}, c_i, d_i, c_u$ and $d_u$, such that if $\inf_{x \in V} \mathbb{P}[Y_x = 1] > p_*$, then $G$ contains almost surely a unique infinite open cluster.
Moreover, fixed any value $\theta > 0$, if the marginal distributions of $\{Y_x\}$ are large enough, then for every site $z\in V$
\begin{equation}
  \label{e:decay_second}
\lim_{v \to \infty}v^\theta \P[v < |\mathcal{C}_z| < \infty] = 0,
\end{equation}
where $\mathcal{C}_z$ denotes the open connected component containing $z$.
\end{Theorem}

We also prove a theorem establishing the existence of a non-trivial sub-critical phase.
This result is simpler to prove but helps to establish a more complete picture of phase transition for dependent percolation on $G$.

\begin{Theorem}\label{thm:p_c>0_dependent}
Let $G$ be a graph satisfying $\mathcal{V}(c_u, d_u)$.
Moreover, let $\P$ be a probability measure that satisfies $\mathcal{D}(\alpha, c_\alpha)$ with $\alpha > \alpha_{\ast \ast}$, where $\alpha_{\ast \ast} > 0$ is defined in Remark~\ref{remark:alpha_double_star}.
Then there exists a $p_{\ast\ast} > 0$, depending only on $G, \alpha, c_\alpha$, such that if $\sup_{x \in V} \mathbb{P}[Y_x = 1] < p_{\ast\ast}$, then the graph contains almost surely no infinite open cluster.
Moreover, fixed $\theta > 0$, if the marginal distributions of $\{Y_x\}$ are small enough, then
\[
\lim_{r \to \infty} r^\theta \P[|\mathcal{C}_z| > r] = 0,
\]
where $\mathcal{C}_z$ denotes the open connected component containing a fixed site $z \in V$.
\end{Theorem}
\begin{Remark}
  \begin{enumerate}[\quad a)]
  \item Note that Theorem~\ref{thm:p_c>0_dependent} does not require $G$ to be roughly transitive.
  \item Moreover, this theorem does not follow from a simple path counting argument because of the dependence present in the law $\mathbb{P}$.
  \item Given the above results, a natural question would be whether the condition $\mathcal{D}(\alpha, c_\alpha)$ on the decay of dependence of $\mathbb{P}$ could be weakened.
    Of course, the parameters $\alpha_{\ast}$ and $\alpha_{\ast \ast}$ that appear above are not supposed to be sharp. However, let us observe that if the exponent $\alpha$ appearing in the decay of dependence of the law $\mathbb{P}$ is slow enough, then there are counterexamples showing that Theorem~\ref{thm:p_c<1_dependent} does not hold, see Subsection~\ref{ss:elipses}.
  \end{enumerate}
\end{Remark}

\subsection{Transitive graphs}

We can specialize our main results to the special case of transitive graphs of polynomial growth.
It is important to observe that the hypothesis \eqref{eq:isoperimetric} is not necessary in this case, since this can be deduced for instance from \cite{CPC:1771424} (cf.\ Appendix~\ref{s:appendix}).
This yields to another consequence of our main result.
\begin{Corollary}\label{thm:p_c<1}
Let $G=(V,E)$ be a transitive graph satisfying $c' r^{d'} \leq |B(o, r)| \leq c'' r^{d''}$ for every $r \geq 1$ and $o\in G$, for some $c', c'' > 0$ and $1 < d' \leq d'' < \infty$. Then $p_c(G) < 1$.
\end{Corollary}
Although the above result was already known, as we discuss in detail in the next subsection, it is worth mentioning that our proof does not make use of Gromov's theorem on groups of polynomial growth, relying instead on probabilistic tools only.

We postpone the proof of the above corollary to Appendix~\ref{s:appendix}.

\paragraph{Previously known results}

Percolation on transitive graphs has been intensively studied in the last decades specially for the independent case.
Let us now mention some of the works that more closely relate to the current article.

In \cite{lyons_1995}, Russell Lyons proved that for independent percolation, $p_c(G) < 1$ if $G$ is a group of exponential growth (see also \cite[Chapter 7]{LP11}).
The case of Cayley graphs of finitely presented groups with one end has been dealt with in \cite{zbMATH01224777} also in the independent case.
A similar question has also been considered on the Grigorchuk group, an example of a group with intermediate growth (see \cite{percolation_Grigorchuk}).
In Corollary~3.2 of \cite{zbMATH06138623} it has been proved that $p_c(G) < 1$ for transitive graphs $G$ satisfying another isoperimetric inequality, see (2.4) and Definition~2.3 of \cite{zbMATH06138623}.
If $G$ is a transitive amenable graph, it was proved in \cite{BK89} that if for some $p$ there exists an infinite open cluster, then it is almost surely unique, see also Theorem~2.4 of \cite{HJ06}.

We also point out that the results proven in \cite{zbMATH06138623} hold in a more general setting than what we describe here.
More precisely, they hold in the case of \emph{quasi-transitive} graphs, but since we do not make use of such graphs in the rest of the paper, we refer the interested reader to \cite{zbMATH06138623} for the details.

A recent work on this topic is \cite{raoufi-yadin}, where the authors look at the percolation threshold for certain groups, which include the so-called \emph{indicable groups}.
(We refer to their paper for the definitions and the precise statements.)
Here we emphasize that indicable groups include groups of polynomial growth.
However, the methods developed in \cite{raoufi-yadin} allow one to work in further generality, for example, with groups of intermediate growth such as the Grigorchuk group (cf.\ \cite[Section 1.3]{raoufi-yadin}).

\vspace{4mm}

The most important relation between previously known results and our work comes at the intersection with Corollary~\ref{thm:p_c<1}, since transitive graphs can be associated with a group of automorphisms, benefiting therefore from important results on group theory.

More precisely, if $G$ is a transitive graph of polynomial growth, then $G$ is quasi-isometric to a Cayley graph of a nilpotent group (see \cite{trofimov}, \cite{losert}, Theorem~4 of \cite{Sabidussi64} or \cite[Theorem 2]{note_authomorphisms}).
This yields two different proofs of Corollary~\ref{thm:p_c<1}.
Let $G$ be the Cayley graph of a nilpotent group with super-linear growth.
Then
\begin{enumerate} [\quad a)]
\item We can use Theorem~7.19 of \cite{LP11} to conclude that there exists a subset of $G$ which is quasi-isometric to $\mathbb{Z}^2$, therefore $p_c(G) < 1$ as desired.
This argument has the advantage that it allows for duality arguments that can work even for dependent percolation.
\item Alternatively, we observe that $G$ is finitely presented (see Exercise~4.3 of \cite{pete_book}) and use Theorem~9 of \cite{zbMATH01224777} to conclude that the number of cut-sets of size $n$ separating a fixed vertex from infinity is at most $c^n$.
Then a simple Peierls-type argument can show that $p_c < 1$.
The added benefit of this approach is that it gives an exponential bound on the probability \eqref{e:decay_second} for Bernoulli percolation on transitive graphs.
\end{enumerate}

\begin{Remark}
  \label{r:advantages}
  In light of the above, let us emphasize some advantages of our approach.
  \begin{enumerate} [\quad a)]
  \item For the case of transitive graphs, our proof does not make use of Gromov's Theorem on groups of polynomial growth.
Although the proof of his original result has been considerably simplified by other authors (cf.\ e.g.\ \cite{kleiner} and \cite{shalom-tao}), Gromov's theorem is quite involved and apparently far from the field of probability.
  \item To the best of our knowledge, the bound in \eqref{e:decay_second} does not seem to follow from the above arguments in the case of dependent percolation on transitive graphs.
  \item Uniqueness of the infinite cluster obtained in Theorem~\ref{thm:p_c<1_dependent} does not depend on the translation invariance of the law $\mathbb{P}$ as is the case with the argument in \cite{BK89}.
  \item Note that being roughly isometric to each other defines an equivalence relation over the class of graphs.
    However, it is important to notice that the distortion constant $\uc{c:rough_trans}$ worsens as we compose rough isometries.
    Therefore, for a given roughly transitive graph there is a priori no analogue of the group of isomorphisms that is fundamental in the case of transitive graphs.
  \item We strongly believe that the techniques we develop here could be easily extended in order to work for weaker notions of transitivity, for example by weakening the notion of rough isometries.
    We however kept the current presentation in order to avoid an overly complicated exposition.
  \end{enumerate}
\end{Remark}


\subsection{Idea of the proofs}

The proofs of Theorem~\ref{thm:p_c<1_dependent} and Corollary~\ref{thm:p_c<1} follow a renormalization scheme which allows us to bound the probability of certain ``bad events'' as the scale size grows.
In this section we will focus on the case of Theorem~\ref{thm:p_c<1_dependent} which is the more elaborate one.

For any $x\in V(G)$ and $L > 0$ set
\begin{equation*}
  \S(x, L) =
  \begin{array}{c}
    \text{``there exist two large connected sets in $B(x, 3L)$,}\\
    \text{which cannot be joined by an open path in $B(x, 3L^2)$''}.
  \end{array}
\end{equation*}
This will play the role of the ``bad event'' in the proof of Theorem~\ref{thm:p_c<1_dependent}, see \eqref{e:separation_event} for a precise definition.

The main advantage of the above event is that it plays two complementary roles.
First, the events $\S(x,L)$ are hierarchical (see the \nameref{lemma:joao}~\ref{lemma:joao}), therefore it is possible to bound their probabilities using inductive arguments coming from a multi-scale renormalization procedure.
Secondly, these events are rich enough that, once we show that $\mathbb{P}[\S(x, L)]$ decays fast as $L$ goes to infinity, we can derive the existence of a unique open infinite connected component, as desired (see Lemma~\ref{lemma:lego}).

For the inductive part of the argument, we need to introduce a rapidly growing sequence $(L_k)_{k \geq 1}$ of scales, see \eqref{eq:inductive_def_L_k}.
As we mentioned above, our objective is to show that for large enough values of the percolation parameter $p$, the probabilies $p_k = \mathbb{P}[\S(o, L_k)]$ of observing a separation event at scale $k$ go to zero fast as $k$ goes to infinity.

\vspace{4mm}

The proof of our main results can then be described through three steps:
\begin{enumerate}[\quad a)]
\item We first show that $\S(o, L_{k+1})$ implies the occurrence of $\S(y_i, L_k)$ for several points $y_i \in B(o, 2L_{k+1}^2)$, see the \nameref{lemma:joao}~\ref{lemma:joao}.
Note that the event $\S(y_i, L_k)$ takes place in the smaller scale $L_k$.
\item Derive from the above a recursive inequality between $p_{k+1}$ and $p_k$, to show that if $p$ is close enough to $1$, then $p_k$ goes to zero fast as $k$ goes to infinity, see Section~\ref{s:reduction}.
\item Finally, in Lemma~\ref{lemma:lego} we show that a fast decay of $p_k$ implies our main result.
\end{enumerate}

Although all of the above steps are essential in establishing  Theorem~\ref{thm:p_c<1_dependent} and Corollary~\ref{thm:p_c<1}, we note that items $b)$ and $c)$ follow the same spirit to what has been done in \cite{2014arXiv1409.5923T}.
For the sake of completeness we also include their proofs in the current paper.
However it is step $a)$ that contains the main novelty of the current work, see the \nameref{lemma:joao}~\ref{lemma:joao}.
It is this lemma that allows us to weaken the \emph{local isoperimetric inequality} of \cite{2014arXiv1409.5923T} to the canonical definition \eqref{eq:isoperimetric} for roughly transitive graphs of polynomial growth.

\subsection{Sketch of the proof of the \nameref{lemma:joao}}

The main new ingredient of this paper is the \nameref{lemma:joao} proved in Section~\ref{s:proof_joao}.
Setting up a renormalization scheme on a graph that is not $\mathbb{Z}^d$ requires a good understanding of the geometry of the graph in question and it is during the proof of \nameref{lemma:joao} that this difficulty is revealed.
For this proof we make strong use of the isoperimetric inequality and rough transitivity of $G$.

The proof of the \nameref{lemma:joao} follows three main steps.
Recall that we are assuming the occurrence of $\S(o, L_{k+1})$, which provides us with two large sets $A^0$, $A^1 \subseteq B(x, 3 L_{k+1})$ which cannot be connected by an open path in $B(x, 3 L_{k+1}^2)$.
Our aim is to show the existence of such separation events in various balls of size $L_k$ inside $B(o, 2L_{k+1}^2)$.
\begin{enumerate}[\quad i)]
\item The first step of the proof will be to reduce the quest of finding separation events $\S(y_i, L_k)$ into simply connecting $A^0$ with $A^1$ through several paths.
This is the content of Lemma~\ref{lemma:new_lemma_3.2}.
\item Therefore, we can assume by contradiction that there exists two sets $A^0$ and $A^1$ which cannot be connected by several paths as above.
However, the isoperimetric inequality \eqref{eq:isoperimetric} guarantees the existence of several disjoint paths (not necessarily open) connecting $A^0$ to distance $L_{k+1}^2$ (similarly for $A^1$), see Lemma~\ref{lemma:N''}.
\item Roughly speaking, in the last step we use the existence of $A^0$ and $A^1$ above in order to embed a binary tree into $G$, which would contradict the polynomial growth of this graph.
We start with the ball $B(o, 3L_{k+1})$ (where the sets $A^0$ and $A^1$ reside) and two paths from the previous step as a building block.
They will respectively represent the root $\varnothing$ of the binary tree and the edges connecting $\varnothing$ to its descendants.
Finally we use the rough transitivity of $G$ to replicate this pattern.
Arguing in a recursive way we obtain the desired embedding, which leads to a contradiction on the polynomial growth of $G$.
\end{enumerate}
Steps $i)$ and $iii)$ are illustrated in Figures~\ref{f:six_balls} and \ref{f:crab_party} respectively.

\bigskip

This paper is organized as follows. In Section~\ref{s:notation} we introduce some preliminary notation and prove an  auxiliary result, followed by Section~\ref{s:proof_pc>0_dep}, where we show Theorem~\ref{thm:p_c>0_dependent}.

In Section~\ref{s:reduction} we define the separation events $S(x, L)$ and state two fundamental intermediate results (Lemmas~\ref{lemma:joao} and \ref{lemma:lego}).
Then, assuming their validity, we prove Theorem~\ref{thm:p_c<1_dependent}, which corresponds to \textit{Step b)} in the outline of the proof of our main results.

Section~\ref{s:proof_joao} is devoted to proving the \nameref{lemma:joao} and is split into three subsections.
Each of these subsections correspond to one step in the above sketch.
Finally we show Lemma~\ref{lemma:lego} in Section~\ref{s:lego}, and we in Section \ref{s:examples} we present some examples of dependent percolation processes for which our results apply.
We conclude with the proof of Corollary~\ref{thm:p_c<1} in Appendix~\ref{s:appendix}.

\subsection*{Acknowledgments}
We are grateful to Yuval Peres, G\'abor Pete, Russell Lyons and Itai Benjamini for bringing to our attention some fundamental references and suggestions.
Thanks also to Mikhail Belolipetski for fruitful discussions.
We are grateful to an anonymous referee for carefully reading this manuscript and making numerous useful suggestions that contributed to improve the paper.

A.T. is grateful to CNPq for its financial contribution to this work through the grants 306348/2012-8 and 478577/2012-5.
This work began during a visit of E.C.\ to IMPA, that she thanks for the support and hospitality.

\section{Notation and auxiliary results}
\label{s:notation}

In this section we introduce some notation and prove some auxiliary results that will be useful throughout the paper.

\subsection{Notation}

For every finite set $A\subset V$ we denote by $|A|$ its cardinality, and by $\partial A$ its edge boundary:
\[
\partial A := \big\{ \{x,y\}\in E \ : \ x\in A \textnormal{ and }y\notin A \big\}.
\]
Analogously, its internal vertex boundary is denoted by
\[
\partial_i A := \big\{ x \in A \ : \text{ there exists $y \in V \setminus A$ such that $\{x, y\} \in E$} \big\}.
\]
For any two vertices $x,y\in V$ we will denote by $d(x,y)$ the \emph{graph distance} between $x$ and $y$, i.e., the minimum number of edges contained in a path that goes from $x$ to $y$.
Analogously, for any two sets $A,B\subset V$ we set
\[
d(A,B):= \min \{ d(a,b) \ : \ a\in A, b\in B\}.
\]
By $B(x, R) $ we denote the ball centered at $x$ and of radius $R \geq 0$ in the graph distance, more precisely, $w\in V$ belongs to $B(x, R) $ if and only if $d(x,w)\leq R$.
Let us define the growth function
\begin{equation}
  \bar{v}_G(r) = \sup_{x \in G} |B(x, r)|,
\end{equation}
where we may omit the sub-index in $v_G$ if it is clear from the context.
\begin{Remark}\label{remark:note_v}
Note that if \eqref{eq:volume_upper_bound} holds, then we have $\bar{v}_G(r)\leq c_u r^{d_u}$.
\end{Remark}
Independent percolation (sometimes called Bernoulli) can be described as follows.
We associate for each vertex $x \in V$ an independent coin toss with success parameter $p \in [0,1]$, in case of success we say that the vertex is \emph{open} otherwise we call it \emph{closed}.
This gives rise to a random sub-graph $\mathbb{G}_p$ of $G$, induced by the set of open vertices.

One of the most interesting features of this model is that for several graphs it presents a phase transition at a critical value $p_c \in (0,1)$.
To make the above statement more precise, we define the critical value $p_c = p_c(G)$ as follows
\begin{equation}
  \label{e:p_c}
  p_c := \sup \{p\in [0,1] \ : \ \P [\textnormal{there exists an infinite cluster on $\mathbb{G}_p$}] = 0\}.
\end{equation}
It follows that, for $p < p_c$, the induced sub-graph contains almost surely only finite connected components, while for $p > p_c$ it contains almost surely at least one infinite cluster.
See \cite{Gri99} for a proof that $p_c \in (0, 1)$ for the case $V = \mathbb{Z}^d$, $d \geq 2$, endowed with edges connecting nearest neighbors vertices.

\subsection{Some remarks about rough isometries}
The results presented here follow the exposition of \cite{elek}, to which the reader is referred for more details.
Suppose that $\phi: G \to G'$ is a $\uc{c:rough_trans}$-rough isometry.
Then for any set $A \subseteq G$ we have
\begin{equation}
  \label{e:large_image}
  |\phi(A)| \geq \frac{|A|}{\bar{v}_G(\uc{c:rough_trans})}.
\end{equation}
In fact, if $d(x, y) \geq \uc{c:rough_trans}$, then $\phi(x) \neq \phi(y)$ by \eqref{e:rough_iso}.
This implies that at most $\bar{v}_G(\uc{c:rough_trans})$ many points can share the same image under $\phi$ in $G'$.

Another interesting property of rough isometries is that they are almost invertible, in the following sense.
\begin{display}
  \label{e:rough_inverse}
  Given a $\uc{c:rough_trans}$-rough isometry $\phi:G \to G'$, there is a
$4\uc{c:rough_trans}^2 $-rough isometry $\psi:G' \to G$ such that $d(x, \phi \circ \psi(x)) \leq \uc{c:rough_trans}$ for any $x \in V$.
\end{display}
Indeed, let us define $\psi(x')$ as the point $x \in V$ such that $d(x', \phi(x))$ is minimized (choosing arbitrarily in case of ties).
First of all, observe by \eqref{e:rough_surj} that $d(x', \phi \circ \psi(x')) \leq \uc{c:rough_trans}$.
We now show that $\psi$ is a $4 \uc{c:rough_trans}^2$-rough isometry and for this fix $x', y' \in G'$.
We can assume that $x' \neq y'$ (the other case is trivial), then one estimates
\[
  \begin{split}
    \frac{1}{4 \uc{c:rough_trans}^2} \; d\big( x', y' \big) - 1 & \leq \frac{1}{4 \uc{c:rough_trans}^2} \Big( d\big ( x', \phi \circ \psi(x')\big )+d\big ( y', \phi \circ \psi(y')\big )+ d\big( \phi \circ \psi(x'), \phi \circ \psi(y') \big) \Big) - 1\\
    &
\stackrel{\eqref{e:rough_surj}}{\leq} \frac{1}{4 \uc{c:rough_trans}^2} \Big( d\big( \phi \circ \psi(x'), \phi \circ \psi(y') \big) + 2 \uc{c:rough_trans} \Big) - 1 \\
& \overset{\eqref{e:rough_iso}}< \frac{1}{4\uc{c:rough_trans}} d\big( \psi(x'), \psi(y') \big)+\frac{1}{\uc{c:rough_trans}}-1\\
    & \stackrel{\uc{c:rough_trans} \geq 1}{\leq} d\big( \psi(x'), \psi(y') \big) \overset{\eqref{e:rough_iso}}\leq \uc{c:rough_trans} d\big( \phi \circ \psi(x'), \phi \circ \psi(y') \big) + \uc{c:rough_trans}\\
    & \leq \uc{c:rough_trans} \Big ( d\big (x',\phi \circ \psi(x')\big )+d\big ( \phi \circ \psi(y'),y'\big ) +d\big( x', y' \big) \Big ) + \uc{c:rough_trans} \\
    & \stackrel{\eqref{e:rough_surj}}{\leq} \uc{c:rough_trans} d\big( x', y' \big) + 2 \uc{c:rough_trans}^2 + \uc{c:rough_trans} \overset{\uc{c:rough_trans} \geq 1}\leq 4 \uc{c:rough_trans}^2 d(x', y').
  \end{split}
\]
Also, if $x'$ belongs to the image of $\phi$, then $d(\phi(\psi(x')), x') = 0$, so that $d(\psi(\phi(x)), x) \leq \uc{c:rough_trans}$, and consequently \eqref{e:rough_surj} also holds for $\psi$.
This concludes the proof of \eqref{e:rough_inverse}

\begin{Remark}
It would be tempting to say that every roughly transitive graph is roughly isomorphic to a transitive one.
This is however not the case, as shown in \cite[Proposition~2]{elek}.
Moreover, the counterexample built in \cite{elek} has indeed polynomial growth, hence implying that our statements cannot be deduced from simple strengthening of previous results.

We would like also to recall Open Question~2.3 of \cite{zbMATH06243708}: ``Is there an infinite $\uc{c:rough_trans}$-roughly transitive graph, which is not roughly-isometric to a homogeneous space, where a homogeneous space is a metric space with a transitive isometry group?''

On the other hand, recall from Remark~\ref{r:advantages} e) that the techniques presented here are believed to work beyond the case of roughly transitive graphs.
\end{Remark}

\subsection{Paving}

For the next lemma, we need also to introduce a lower bound on the volume growth of balls on $G$.
\begin{Definition}
Given constants $c_l, d_l > 0$, we say that $G$ satisfies $\mathcal{L}(c_l, d_l)$ if for every \emph{real} number $r \geq 1$ and every site $x \in V$
\begin{equation}
  \label{eq:volume_lower_bound}
  |B(x,r)| \geq c_l r^{d_l}.
\end{equation}
\end{Definition}
Note that every infinite connected graph satisfies the above bound for $d_l = 1$ and we don't need more than this for our proofs.
However, if one knew in advance that the above condition holds for some $d_l > 1$, the final results will be improved through a smaller $\alpha_{\ast}$ or $\alpha_{\ast \ast}$, see \eqref{e:alpha_star}.

Proposition~\ref{lemma:paving} below allows us to cover a large ball of radius $r^2$ with smaller balls of radius $s$.
This can be thought of as a replacement for paving arguments for renormalization procedures on the lattice $\mathbb{Z}^d$.

In the following, for any set of vertices $K\subset V$, define
\[
B(K, s):= \bigcup_{y\in K}B(y,s).
\]
\nc{c:paving}
\begin{Proposition}
  \label{lemma:paving}
  If $G = (V, E)$ satisfies the volume growth estimates $\mathcal{V}(c_u, d_u)$ and $\mathcal{L}(c_l, d_l)$, then there is a constant $\uc{c:paving} = \uc{c:paving}(c_l, d_l, c_u, d_u)$ such that
  \begin{display}
    \label{eq:paving}
    for every $r \geq 1$ and $ s\in [ 2 , 2r^2]$, for every $x\in V$, there exist $K \subseteq B(x, 2r^2)$, \\
    such that $B(x, r^2) \subseteq B(K, s)$ and $|K| \leq \uc{c:paving}\frac{r^{2d_u}}{s^{d_l}}$.
  \end{display}
\end{Proposition}

\begin{proof}
Fix $s$ in the range given in the hypothesis and take the set $K \subseteq B(x, 2r^2)$ to be an arbitrary \emph{maximal} set satisfying
\begin{display}
  \label{e:mutually_far}
  $d(y, y') \geq s$ for every $y, y' \in K$.
\end{display}
Since $K$ is maximal, it is also an $s$-net of $B(x, 2r^2)$, or in other words $B(x, 2r^2) \subseteq B(K, s)$.
By \eqref{e:mutually_far}, all the balls $\{B\bigl (y,  s/3 \bigr )\}_{y\in K}$ are disjoint.
Therefore, by the lower bound $\mathcal{L}(c_l, d_l)$ we obtain
\[
\bigl |B(K,  s/3 )\bigr |=\sum_{y\in K} \bigl |B(y,  s/3)\bigr | \geq |K| c_l \, \frac{ s^{d_l}}{3^{d_l} }.
\]
On the other hand, $|B(K, s )|\leq |B(x, 2r^2+s)|\leq c_u (2r^2+s)^{d_u}\leq c_u (4r^2)^{d_u}.$
By putting together these two facts, we obtain that there is a positive constant $\uc{c:paving}=\uc{c:paving}(c_l,d_l,c_u,d_u)$ such that
\[
|K|\leq \uc{c:paving} \frac{r^{2d_u}}{s^{d_l}}.
\]
The above argument implies that there exists a set $K\subseteq B(o, 2r^2)$ such that the statement holds.
\end{proof}

\subsection{Decoupling several events}

Our next statement is a consequence of the decoupling inequality from Definition~\ref{def:decoupling}.

\begin{Proposition}\label{claim:lemma4.2}
Suppose that $\P$ satisfies the decoupling inequality $\mathcal{D}(\alpha,c_\alpha)$ for some $\alpha>0$.
Now fix any value of $r\geq 1$, an integer $J'\geq 2$ and distinct points $y_1, y_2, \ldots, y_{J'} \in V$ such that
\[
\min_{1 \leq i < j \leq J'} d(y_i, y_j)\geq 3r.
\]
Then for any set of events $\mathcal{G}_1, \ldots, \mathcal{G}_{J'} $ such that $ \mathcal{G}_i\in \sigma (Y_z,  z \in B(y_i, r))$ we have
\begin{equation}
  \label{e:decouple_various}
  \P\bigl ( \mathcal{G}_1\cap \ldots \cap \mathcal{G}_{J'}\bigr ) \leq \bigl ( \P(\mathcal{G}_1)+c_\alpha r^{-\alpha}\bigr )\dots \bigl ( \P(\mathcal{G}_{J'})+c_\alpha r^{-\alpha}\bigr ).
\end{equation}
\end{Proposition}
\begin{proof}
The proof is immediate from Definition \ref{def:decoupling}.
In fact, setting $\mathcal{G}' = \mathcal{G}_1 \cap \dots \cap \mathcal{G}_{J' - 1}$,
\begin{equation}
\P\bigl ( \mathcal{G}_1\cap \ldots \cap \mathcal{G}_{J'}\bigr ) \overset{\mathcal{D}(\alpha, c_\alpha)}\leq \left (\P ( \mathcal{G}_{J'})+c_\alpha r^{-\alpha}\right  ) \P\bigl ( \mathcal{G}_1\cap \ldots \cap \mathcal{G}_{J'-1}\bigr ).
\end{equation}
By iterating this calculation, we obtain the statement.
\end{proof}
\begin{Remark}\label{remark:c-alpha}
Note that in the above lemma we allow $c_\alpha$ to depend on the value $J'$.
\end{Remark}
\begin{Remark}
Here we emphasize that all throughout the paper we always make use of \eqref{e:decouple_various}, which is implied by Definition~\ref{def:decoupling}.
\end{Remark}

\section{Proof of  Theorem~\ref{thm:p_c>0_dependent}}\label{s:proof_pc>0_dep}

This proof is inspired by previous renormalization procedures that were developed for $\mathbb{Z}^d$, see for instance \cite{Szn09}.
Here we adapt them to work on more general classes of graphs.
Although Theorem~\ref{thm:p_c>0_dependent} is not the central result of the current article, we present its proof before for two reasons.
First, it is a warm-up to the proof of Theorem~\ref{thm:p_c<1_dependent} and secondly, it includes some lemmas that will be useful later in the text.
\begin{Remark}
We remark here that, for convenience, we will prove the result on the \emph{diameter} of the largest component, which is equivalent to the previous statement since $\theta$ is arbitrary.
\end{Remark}
Let us first define what we call the crossing event
\begin{equation}
  \label{e:crossing_events}
  \mathcal{T}(x, L) =
  \bigg[
  \begin{array}{c}
    \text{there is an open path from $B(x,3L)$ to $\partial B(x, 3L^2)$}
  \end{array}
  \bigg].
\end{equation}

Our main argument shows the decay of the probabilities of $\mathcal{T}(x, L)$ following a renormalization scheme.
This procedure relates the probabilities of the above events at different scales, that we now introduce.

Given some integer $\gamma \geq 3$, we set
\begin{equation} \label{eq:inductive_def_L_k}
  L_0:= 10 000, \textnormal{ and } L_{k+1} = L_k^\gamma, \ \textnormal{ for all } k \geq 0.
\end{equation}

\begin{Remark}
We have not yet chosen $\gamma$ because it will assume different values for the proofs of Theorems~\ref{thm:p_c<1_dependent} and \ref{thm:p_c>0_dependent}, see Remarks~\ref{remark:alpha_double_star} and \ref{remark:alpha_star} below.
\end{Remark}

In the next definition we introduce the concept of a \emph{cascading} family of events.
Intuitively speaking, it means that if some event occurs at a given scale $L_{k+1}$, then it must also occur several times in the previous scale $L_k$ in well separated regions.

\nc{c:k_def_joao}
\begin{Definition}
  \label{d:joao}
  We say that a family of events $\big( \mathcal{E}(x, L_k) \big)_{x \in V, k \geq 1}$ is \emph{cascading} if for any $J \geq 1$ there exists $\uc{c:k_def_joao} = \uc{c:k_def_joao}(G, J, \gamma)$ for which the following holds.
  Fix any $x \in V$, $k \geq \uc{c:k_def_joao}$ and set $K \subseteq B(x, 2L_{k+1}^2)$ such that $B(K, L_k)$ covers $B(x, L_{k+1}^2)$.
  Then, if the event $\mathcal{E}(x,L_{k+1})$ occurs, there exists a sequence
  \begin{equation}
    y_1, y_2, \ldots ,y_J \in K \text{ with $d(y_j, y_l) \geq 9L_k^2$ for all $j \neq l$}
  \end{equation}
  and such that $\mathcal{E}(y_j, L_k)$ occurs for all $j \leq J$.
\end{Definition}

The importance of the above definition is that it allows us to relate the probabilities of events $\mathcal{E}$ at different scales using recursive inequalities together with the decoupling provided by $\mathcal{D}(\alpha, c_\alpha)$.

\nc{c:cross_cas}
\begin{Lemma}
  \label{l:cross_cas}
  The family of events $\{\mathcal{T}(x, L)\}$ defined in \eqref{e:crossing_events} is cascading in the sense of Definition~\ref{d:joao}.
\end{Lemma}

\begin{proof} 
  We first fix $J \geq 1$ and let $\uc{c:cross_cas} \geq 1$ be such that for all $k\geq \uc{c:cross_cas}$ we have
  \begin{equation}
    \label{e:choose_c_cross_cas}
    3L_{k+1} + 30J L_k^2 \leq L_{k+1}^2,
  \end{equation}
  which can be done by our choice of scales in \eqref{eq:inductive_def_L_k}.

  To prove that the events $\mathcal{T}(x, L)$ are cascading, let us pick $k \geq \uc{c:cross_cas}$, $x \in V$ and assume that $\mathcal{T}(x, L_{k+1})$ occurs, that is
  \begin{display}
    there exists an open path $\sigma$ from $B(x, 3L_{k+1})$ to $\partial B(x, 3L_{k+1}^2)$.
  \end{display}
  Let us consider the concentric spheres $S_j = \partial B(x, 3L_{k+1} + (30j)L_k^2)$, for $j = 1, \dots, J$.
  Note that all these spheres are contained in $B(x, L_{k+1}^2)$ by \eqref{e:choose_c_cross_cas}.

  We now let $x_j$ be the first point of intersection of the path $\sigma$ to $S_j$.
  Given the set $K$ as in Definition~\ref{d:joao} (or more precisely, such that $ K\subset B(x, 2L_{k+1}^2)$ and $B(K, L_k)$ covers $B(x, L_{k+1}^2)$) we can pick $y_j \in K$ ($j \leq J$) such that $x_j \in B(y_j, L_k)$.

  We see that the distance between two distinct $y_j$'s is at least
  \begin{equation}
    d(y_j, y_{j'}) \geq d(x_j, x_{j'}) - 2L_k \geq 30L_k^2 - 2L_k \geq 9L_k^2
  \end{equation}
  as required in Definition~\ref{d:joao}.
  To finish the proof, observe that the open path $\sigma$ that guarantees the occurrence of $\mathcal{T}(x, L_{k+1})$ can be split into pieces that show the occurrence of $\mathcal{T}(y_j, L_k)$, for $j \leq J$.
  The piece corresponding to $j$ can be constructed for instance by picking the first time $\sigma$ touches $x_j$ until it first exits $B(y_j, 3L_k^2)$.
  This finishes the proof of the lemma.
\end{proof}

\begin{Remark}
  In the next section we will turn to the proof of Theorem~\ref{thm:p_c<1_dependent} and for this we define another family of events (denoted by $\mathcal{S}(x,L)$) and prove a result which is analogous to Lemma~\ref{l:cross_cas}, namely the \nameref{lemma:joao}.
  However, the proof that the events $\mathcal{S}(x, L)$ are cascading will be more involved.

  It is important to observe that some definitions and arguments in this section were written in such a way that they can be used also during the proof of Theorem~\ref{thm:p_c<1_dependent}, instead of optimizing for brevity.
\end{Remark}

The importance of the definition of cascading events will become clear in the following bootstrapping result.
Given a scale sequence $L_k$ as in a family of events $\mathcal{E}(x, L_k)$ (each event must be measurable with respect to what happens in $B(x,3L_k^2)$) let
\begin{equation}
  p_k^{\mathcal{E}} = \sup_{x \in V} \mathbb{P}(\mathcal{E}(x, L_k)).
\end{equation}

\nc{c:cascade_decays}
\begin{Lemma}
  \label{l:cascade_decays}
  Suppose that $G$ satisfies $\mathcal{V}(c_u, d_u)$ and $\mathcal{L}(c_l, d_l)$ and $\mathbb{P}$ has the decoupling inequality $\mathcal{D}(\alpha, c_\alpha)$, for $\alpha > \gamma d_u - d_l/2$.
  Moreover, let $\mathcal{E}(x, L_k)$ be a family of events which is cascading in the sense of Definition~\ref{d:joao}, then for any $\beta \geq 2 \alpha$, there exists a constant $\uc{c:cascade_decays} = \uc{c:cascade_decays}(\beta, c_i, d_i, c_u, d_u, c_\alpha, \alpha, \uc{c:k_def_joao}) \geq 1$ such that
  \begin{display}
    \label{e:induct_cascade}
    if for some $k_o \geq \uc{c:cascade_decays}$ we have $p^{\mathcal{E}}_{k_o} \leq L_{k_o}^{-\beta}$ then $p^{\mathcal{E}}_{k} \leq L_{k}^{-\beta}$ for all $k \geq k_o$.
  \end{display}
\end{Lemma}

\begin{Remark}
  We note that the constant $\uc{c:k_def_joao}$ depends on the events under consideration.
  For Lemma~\ref{lemma:joao}, where we apply Lemma~\ref{l:cascade_decays}, we make explicit the dependencies of $\uc{c:k_def_joao}$.
\end{Remark}

\begin{proof}
Given $\beta \geq 2 \alpha $, pick an integer $J'$ such that
\[
J'\geq \max \left \{ 2, \frac{\gamma \beta + 1}{2 \alpha -(2 \gamma d_u - d_l)}\right \}.
\]
which is possible since $2 \alpha > 2 \gamma d_u - d_l$.

We can now apply Proposition~\ref{lemma:paving} and for any fixed $k\geq 1$ set $s:=L_k$ and $r:=L_{k+1}$, which gives us a set $K\subseteq B(o, 2L_{k+1}^2)$ such that
\[
|K|\leq \uc{c:paving}L_k^{2\gamma d_u-d_l} \qquad \text{and} \qquad B(o, L_{k+1}^2)\subseteq B(K,L_k).
\]
Our purpose is to bound the probabilities $p^{\mathcal{E}}_k$ using induction.
In fact, using the fact that the events $\mathcal{E}(x, L_k)$ are cascading, for all $k$ large enough we have:
\[
\begin{split}
p^{\mathcal{E}}_{k+1} & \leq \P \big[\exists \, y_1, \ldots , y_{J'}\in K\text{ at mutual distance at least }9L_k^2,\text{ s.t.\ }\mathcal{E}(y_i,L_k)\text{ occurs }\forall i\leq J' \big]\\
& \stackrel{\text{Prop.~}\ref{claim:lemma4.2}}{\leq} \left ( \uc{c:paving}L_k^{2\gamma d_u-d_l} \right )^{J'}\left ( p^{\mathcal{E}}_k+c_\alpha L_k^{-2 \alpha}\right )^{J'}.
\end{split}
\]
Assume as in \eqref{e:induct_cascade} that for some $k_0$ large enough we have $p^{\mathcal{E}}_{k_0} \leq L_{k_0}^{-\beta}$, we need to show that this condition holds for all $k \geq k_0$.
In fact, by using the fact that
\begin{equation}
( p^{\mathcal{E}}_{k_0}+c_\alpha L_{k_0}^{-2 \alpha} ) \leq (c_\alpha +1) L_{k_0}^{-\min\{2 \alpha,\beta\}} \overset{\beta \geq 2 \alpha}{\leq} (c_\alpha + 1) L_{k_0}^{-2 \alpha},
\end{equation}
we obtain:
\[
\frac{p^{\mathcal{E}}_{k_0+1}}{L_{k_0+1}^{-\beta}}\leq  \uc{c:paving}^{J'}(c_\alpha+1 )^{J'}L_{k_0}^{J'(2\gamma d_u-d_l)- 2 J' \alpha + \gamma \beta} \leq \uc{c:paving}^{J'}(c_\alpha+1 )^{J'}L_{k_0}^{-J'(2 \alpha -(2\gamma d_u-d_l))+\gamma \beta}.
\]
Note that we have chosen $J'$ such that the exponent $-J'(2 \alpha -(2\gamma d_u-d_l))+\gamma \beta$ is smaller or equal to $-1$, making the RHS above smaller than $1$ for all $k$ large enough, proving \eqref{e:induct_cascade} for $k = k_0 + 1$.
We can now continue the proof for every $k \geq k_0$ using induction.
\end{proof}

\begin{Remark}\label{remark:alpha_double_star}
Recall that in Theorem~\ref{thm:p_c>0_dependent} we have used the value $\alpha_{\ast \ast}$ without giving its precise value.
We can now introduce
\begin{equation}
  \label{e:alpha_double_star}
  \alpha_{\ast \ast} := 3 d_u - d_l/2.
\end{equation}
\end{Remark}

\begin{proof}[Proof of Theorem~\ref{thm:p_c>0_dependent}]
We now fix $\gamma = 3$ and let the scale sequence $(L_k)_{k \geq 0}$ be defined as in \eqref{eq:inductive_def_L_k}.
Observe also that for $\alpha > \alpha_{\ast \ast}$ as in \eqref{e:alpha_double_star}, we have $\alpha > \gamma d_u - d_l/2$ as required in Lemma~\ref{l:cascade_decays}.

Therefore, we are in position to apply Lemma~\ref{l:cascade_decays} for some arbitrarily chosen $ \beta > \max\{2\alpha, 6 \theta\}$.
In order to show that $p_k \leq L_k^{-\beta}$ for large enough $k$, we have simply to show that $p_{k_o} \leq L_{k_o}^{-\beta}$ for some $k_o \geq \uc{c:cascade_decays}$.

But by a simple union bound,
\begin{equation}
  p^{\mathcal{T}}_{k_o} = \sup_{x \in V} \mathbb{P}(\mathcal{T}(x, L_{k_o})) \leq \sup_{x\in V} \mathbb{P} \big[ Y_z = 1 \text{ for some } z \in B(x, L_{k_o}) \big] \leq c_{d_u} L_{k_o}^{d_u} \sup_{x \in V}[Y_x = 1].
\end{equation}
Therefore, as soon as
\begin{equation}
  \sup_{x \in V} \mathbb{P}[Y_x = 1] \leq \frac{1}{c_{d_u}} L_{k_o}^{- d_u - \beta},
\end{equation}
we have $p^{\mathcal{T}}_{k_o} \leq L_{k_o}^{-\beta}$ as desired and therefore $p_k \leq L_k^{-\beta}$ for all $k \geq k_o$.

To finish, given a large enough $r \geq 1$, take $\bar{k}$ such that $3L_{\bar{k}}^2 \leq r < 3L_{\bar{k}+1}^2$.
Then,
\begin{equation}
  r^\theta \P[\diam(\mathcal{C}_o) > r] \leq 3L_{\bar{k} + 1}^{2\theta} p^{\mathcal{T}}_{\bar{k}} \leq L_{\bar{k}}^{2\gamma \theta - \beta}
\end{equation}
The proof now follows from the fact that $\gamma=3 $ and $\beta > 6 \theta$.
\end{proof}

\section{Proof of Theorem~\ref{thm:p_c<1_dependent}}
\label{s:reduction}

The proof of Theorem~\ref{thm:p_c<1_dependent} follows the same lines of the previous section.
Again, for convenience, we will show the result on the diameter of the largest component, which is equivalent to the original statement, since $\theta$ is arbitrary.

We are going to define a family of events $\mathcal{S}(x, L)$ and then show that they are cascading in the sense of Definition~\ref{d:joao}.
This task will however be much more involved than in the previous section.

We now define what we call a \emph{separation event}.
This will play the role of a ``bad'' event whose probability we intend to bound from above.
Roughly speaking, the separation event says that inside a big ball one can find two large and separated clusters (which are not necessarily open).

Denoting by $\diam(Y)$ the diameter of the set $Y$, for every $x\in V$ and $L\in \R_+$, the separation event $\S(x,L)$ is defined as follows:
\begin{equation}
  \label{e:separation_event}
  \S(x,L):=
  \left [
  \begin{array}{c}
    \text{there are disjoint connected sets $A^0,A^1\subseteq B(x,3L)$}\\
    \text{with $\diam(A^i) \geq L/100$, such that there is no open path}\\
    \text{ contained in $B(x,3L^2)$ connecting $A^0$ with $A^1$}.
  \end{array}
  \right ].
\end{equation}
See Figure~\ref{f:six_balls} for an illustration of the above event.

\begin{Remark}
  \label{r:monotone}
  Observe that the above defined event is decreasing in the sense that if $\mathcal{S}(x, L)$ occurs and we close more vertices in $G$, then $\mathcal{S}(x, L)$ will also occur.
\end{Remark}

Recall the definitions of $d_i$ and $d_u$ from \eqref{eq:isoperimetric} and \eqref{eq:volume_upper_bound} respectively, and consider a fixed integer $\gamma \geq 3$ such that
\begin{equation}\label{eq:condition_gamma}
  \gamma \left ( \frac{d_i-1}{d_i}\right )>2d_u.
\end{equation}
Note that this is legitimate, by the assumptions of Theorem~\ref{thm:p_c<1_dependent}.
As above we set
\begin{equation}\label{eq:scales_new}
  L_0:= 10 000, \textnormal{ and } L_{k+1}=L_k^\gamma, \ \textnormal{ for all } k\geq 0.
\end{equation}

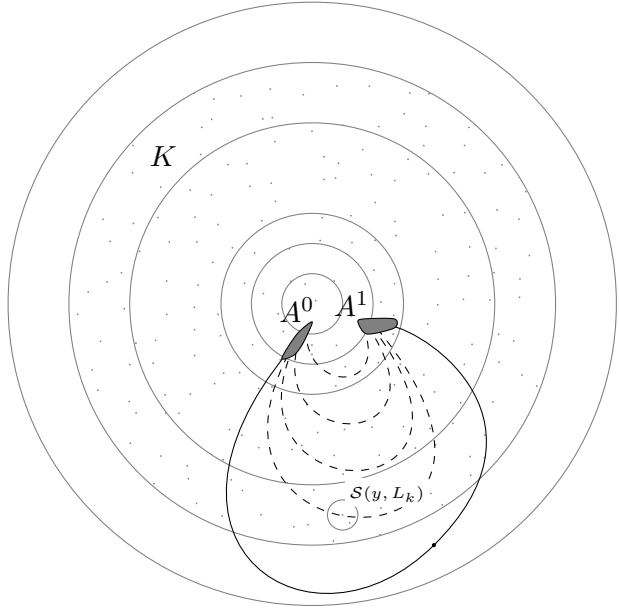
\begin{figure}[ht]
  \centering
  \begin{tikzpicture}[scale=0.4]
    \begin{scope}
      \clip (6,6) circle (8);
      \foreach \x in {1,...,16}
      { \foreach \y in {1,...,16}
        \draw[color=gray,fill=gray] (1 * \x + 0.4 * rand - 3, 1 * \y + 0.4 * rand - 3) circle (0.01);
      }
    \end{scope}
    \draw[fill, color=white] (.6,10.3) rectangle (1.5,11.3);
    \node at (1.1,10.9) {$K$};
    \draw[fill, color=white] (6.5,3.1) rectangle (7.5,2.1);
    \foreach \x in {1,...,3} { \draw[color=gray] (6, 6) circle (\x); }
    \foreach \x in {6,8,10} {\draw[color=gray] (6, 6) circle (\x);}
    \foreach \x in {0, 1, 2, 3}
    { \draw[dashed] (5.8 - 0.2 * \x, 5.2 - 0.3 * \x) .. controls (5.7 - 1.4 * \x, 3 - 2 * \x) and 
                    (8.7 + 2 * \x, 3 - 2 * \x) .. (7.6 + 0.2 * \x, 5.4 - 0.1 * \x);
    }
    \draw[color=gray] (7,-1) circle (.5);
    \draw[fill, color=gray] (7,-1) circle (.01);
    \draw[fill, color=white] (7.1,-.7) rectangle (9.8,.2);
    \node[below right] at (6.9,0.3) {\tiny $\mathcal{S}(y,L_k)$};
    \draw (5.2,4.4) .. controls (0, -2) and (6,-6) .. (10,-2) .. controls (14,2) and (10, 5) .. (8.5,5.3);
    \draw[fill] (10,-2) circle (0.05);
    \draw[fill=gray] plot [smooth cycle] coordinates {(6,5.4) (5.5,5) (5,4.2) (5.4,4.3) (5.7, 4.7)};
    \draw[fill=gray] plot [smooth cycle] coordinates {(7.5,5.4) (7.8,5) (8.7,5.2) (8.7,5.5) (7.7, 5.5)};
    \node[left] at (6.4,5.8) {$A^0$};
    \node[left] at (8.2,6) {$A^1$};
  \end{tikzpicture}
  \caption{The six balls $B(x,L_{k+1})$, $B(x,2L_{k+1})$, $ B(x,3L_{k+1})$ and $B(x,L_{k+1}^2)$, $ B(x,2L_{k+1}^2)$, $ B(x,3L_{k+1}^2)$.
    The sets $A^0$ and $A^1$ from the definition of $\S(x,L_{k+1})$ are pictured, together with a solid path connecting them.
    According to the definition of $\S(x,L_{k+1})$, this solid path must pass through a closed vertex.
    The gray dots in the picture represent the set $K$ from Proposition~\ref{lemma:paving}.
    We also indicate the occurrence of the event $\mathcal{S}(y,L_k)$ as in Lemma~\ref{lemma:new_lemma_3.2}.
  }
  \label{f:six_balls}
\end{figure}

\noindent
By $p_k$ we denote the probability to observe a separation event at scale $k$, i.e., set
\begin{equation}
  \label{eq:pk}
  p_k := \sup_{x \in V}\mathbb{P}[\S(x, L_k)].
\end{equation}
In the above definition we use the supremum over $x \in V$, as we are not necessarily assuming that $G$ is transitive or that $\mathbb{P}$ is translation invariant.

A fundamental step in the proof of Theorem~\ref{thm:p_c<1_dependent} is to show that for values of $p$ close enough to one, the probabilities $p_k$ decay to zero very fast as $k$ increases.

In this section we assume that $G$ satisfies the extra Condition~\ref{cond:joao} below and prove Theorem~\ref{thm:p_c<1_dependent}.
This condition will later be proved to hold true for roughly transitive graphs satisfying $\mathcal{V}(c_u, d_u)$ and $\mathcal{I}(c_i, d_i)$ with $d_i > 1$, see the \nameref{lemma:joao}~\ref{lemma:joao}.

Roughly speaking Condition~\ref{cond:joao} states that if $\S(o,L_{k+1})$ occurs for some $k+1$, then we can find various separation events at the smaller scale $k$.

\begin{Condition}
  \label{cond:joao}
  We say that a given graph satisfies Condition~\ref{cond:joao} for some integer $\gamma \geq 3$ and $L_k$ as in \eqref{eq:scales_new} if the collection of events $\big( \S(x, L_k) \big)_{x \in V, L_k \geq 1}$ is cascading in the sense of Definition~\ref{d:joao}.
\end{Condition}

Before proceeding, let us briefly recall how a statement similar to the above was derived in \cite{2014arXiv1409.5923T} and the main challenges that we face in our context.
In that paper, a stronger hypothesis on the underlying graph was assumed, namely that $G$ verifies certain \emph{local isoperimetric inequalities}.

In the current work, we only make use of the standard isoperimetric inequality \eqref{eq:isoperimetric}, together with the hypothesis that $G$ is roughly transitive and has polynomial growth (see also Remark~2.3 (c) in \cite{2014arXiv1409.5923T}).
In particular, the next lemma will guarantee that Condition~\ref{cond:joao} is implied by $\mathcal{V}(c_u, d_u)$ and $\mathcal{I}(c_i, d_i)$, with $d_i > 1$.
This will be an important novelty of this work and we will postpone its proof to Section~\ref{s:proof_joao}.

\begin{Lemma}[Cascading Lemma]
  \label{lemma:joao}
  Let $G$ be $\uc{c:rough_trans}$-roughly transitive, satisfying the conditions $\mathcal{V}(c_u, d_u)$, \eqref{eq:isoperimetric} with $d_i > 1$, and let $\gamma$ be as in \eqref{eq:condition_gamma}, then $G$ satisfies Condition~\ref{cond:joao}.
  Moreover, the constant $\uc{c:k_def_joao}$ appearing in Definition~\ref{d:joao} depends only on $\uc{c:rough_trans}, J, c_i, d_i, c_u$ and $d_u$.
\end{Lemma}

We will now give a proof of Theorem~\ref{thm:p_c<1_dependent}, assuming the validity of the \nameref{lemma:joao} above, which will be proved in Section~\ref{s:proof_joao}.

\begin{Remark}\label{remark:alpha_star}
In Theorem~\ref{thm:p_c<1_dependent}, we assumed that $\alpha > \alpha_{\ast}$, which still had to be defined.
We can now introduce
\begin{equation}
  \label{e:alpha_star}
  \alpha_{\ast} := \left ( \frac{2 d_u d_i}{d_i-1}\right ) d_u - d_l/2,
\end{equation}
Note that for $\alpha > \alpha_{\ast}$
\begin{display}
  we can find $\gamma$ as in \eqref{eq:condition_gamma} and such that $\alpha > \gamma d_u - d_l/2$ as in Lemma~\ref{l:cascade_decays}.
\end{display}
\end{Remark}

\bigskip

Recall the definition of $p_k$ from \eqref{eq:pk}.
We first show the decay of $p_k$ for large enough $p$ in the following lemma.

\nc{c:k_pk_decay}
\begin{Lemma}
  \label{lemma:pk_decay}
  Suppose that $G$ is a roughly transitive graph satisfying $\mathcal{V}(c_u, d_u)$, $\mathcal{L}(c_l, d_l)$ and $\mathcal{I}(c_i, d_i)$ with $d_i > 1$, $\gamma$ satisfies \eqref{eq:condition_gamma} and fix $\mathbb{P}$ fulfilling $\mathcal{D}(\alpha, c_\alpha)$.
  Then, given any $\beta > 0$, there exists $p_* = p_*(\beta, \gamma, c_i, d_i, c_l, d_l, c_u, d_u, c_\alpha, \alpha, \uc{c:rough_trans}) < 1$ such that whenever $p:=\inf_{x\in V}Y_x > p_*$ we have
  \begin{equation}
    p_k \leq L_k^{-\beta}, \text{ for every $k \geq 1$}.
\end{equation}
\end{Lemma}

\begin{proof}
Since $\gamma$ satisfies \eqref{eq:condition_gamma}, we can use the \nameref{lemma:joao} to conclude that the events $\mathcal{S}(x, L)$ are cascading.
Hence whenever $\beta \geq 2\alpha$, Lemma~\ref{l:cascade_decays} implies that if for some large value $k_0$ we have $p_{k_0} \leq L_{k_0}^{-\beta}$, then this relation holds for all $k \geq k_0$.

We now observe that as the percolation parameter $p$ converges to one, then the probability of $\S(o, L_k)$ (for some fixed $k \geq \uc{c:k_pk_decay}$) converges to zero, since balls will likely be completely open.
This implies that, if $\inf_{x \in V} \mathbb{P}[Y_x = 1] \geq p_*$ as in the statement of the theorem, we will have $p_k \leq L_k^{-\beta}$ for all $k \geq \uc{c:k_pk_decay}$.
By possibly increasing $p_*$ we can make sure that the above holds for all $k \geq 1$ and the value of $\beta \geq 2\alpha$ chosen above.
Note however that taking $\beta$ larger can only make the statement harder to prove, finishing the proof for every $\beta > 0$.
\end{proof}

The statement of Theorem~\ref{thm:p_c<1_dependent} now follows from Lemma~\ref{lemma:pk_decay} above and the following result, whose proof is deferred to Section \ref{s:lego}.
\begin{Lemma}
  \label{lemma:lego}
Suppose that $G$ satisfies \eqref{eq:volume_upper_bound}.
Fix an arbitrary value $\theta > 0$, an integer $\gamma \geq 3$ satisfying \eqref{eq:condition_gamma} and 
$\beta > \gamma(1+\theta)$.
If $p_k \leq L_k^{-\beta}$ for all $k\geq 1$, then
  \begin{equation}
    \mathbb{P} \big[ \text{there is a unique infinite connected open cluster $\mathcal{C}_\infty$} \big] = 1
  \end{equation}
  and moreover, for every fixed $x\in V$ and $L$ large enough we have
  \begin{equation}
 \P[L < |\mathcal{C}_x| < \infty] \leq L^{-\theta},
  \end{equation}
  where $\mathcal{C}_x$ stands for the open connected component containing $x$.
\end{Lemma}
Note the similarity between the above result and Lemma~4.1 of \cite{2014arXiv1409.5923T}.
It is worth mentioning that despite this similarity, a new proof of the above lemma is required since the definitions of $\S(o, L_k)$ and consequently of $p_k$ are different.

\begin{proof}[Proof of Theorem~\ref{thm:p_c<1_dependent}]
  Choose $\beta >\max\{2\gamma d_u-d_l,\gamma(1+\theta)\}$ for some arbitrary value $\theta>0$ and suppose that $G$ is a roughly transitive graph satisfying $\mathcal{V}(c_u, d_u)$ and $\mathcal{I}(c_i, d_i)$ with $d_i > 1$.
  Then we are in the condition to apply Lemma~\ref{lemma:pk_decay}, obtaining that $ p_k\leq L_k^{-\beta}$ for all $k\geq 1$.
  Finally the result follows from Lemma~\ref{lemma:lego}.
\end{proof}

\section{Proof of the \nameref{lemma:joao}}
\label{s:proof_joao}

As we mentioned above, the most innovative step in proving Corollary~\ref{thm:p_c<1} was the intermediate \nameref{lemma:joao}, that we now prove.
The argument is split into three main steps that can be informally described as follows.

\emph{Step 1.} Suppose we have two sets $A^0, A^1$ which are separated as in the definition of $ \S(o,L_{k+1})$.
We first show that paths connecting $A^0$ to $A^1$ necessarily cross a separation event at the smaller scale $L_k$.
This is explained in Subsection~\ref{ss:path}.

\emph{Step 2.} Therefore our task is now reduced to showing that there are several paths connecting these two sets inside $ B(o, 3L_{k+1}^2)$.
This is not an immediate consequence of the isoperimetric inequality \eqref{eq:isoperimetric}.
However, this inequality shows that there must be several disjoint paths connecting $A^0$ to $\partial_i B(o,3L_{k+1}^2)$ (same for $A^1$), see Subsection~\ref{ss:disjoint_paths}.

\emph{Step 3.} Finally, we will show that indeed there exist several paths connecting $A^0$ to $A^1$ and this is done by contradiction.
More precisely, assuming that there are only few paths connecting these sets, we have a type of ``local bottleneck'' in our graph.
This, together with rough transitivity will allow us to replicate this local bottleneck in different parts of the graph and they act as branching points for paths of the graph.
Therefore, we are able (under this contradiction assumption) to embed a chunk of a binary tree inside $G$, see Figure~\ref{f:crab_party}.
This will contradict the polynomial growth that we assumed in first place, concluding the proof of the \nameref{lemma:joao}.
This final argument can be found in Subsection~\ref{ss:embed}.

\subsection{Using paths to find separation events}
\label{ss:path}

The first step in the proof of the \nameref{lemma:joao} is to reduce the task of finding separation events at the finer scale $k$ to simply finding paths between the separated sets $A^0$ and $A^1$ at scale $L_{k+1}$.

First, we observe that, given the inductive definition of $L_k$ in \eqref{eq:inductive_def_L_k}, for all $k \geq 0$ we have
\begin{equation}
  L_k \leq \frac{L_{k+1}}{2000}.
\end{equation}
The next lemma helps us obtaining separation events from paths connecting $A^0$ to $A^1$.

A path connecting $A^0$ to $A^1$ is a (not necessarily open) sequence of adjacent edges that goes from a vertex of $A^0$ to a vertex in $A^1$.
We say that such a path is open if all the vertices visited by the path are open except for its endpoints.

\begin{Lemma}\label{lemma:new_lemma_3.2}
For some $x \in V$ and any $k\geq 0$, consider a set $K \subseteq B(x, 2L_{k+1}^2)$ such that $B(K, L_k) $ covers $ B(x, L_{k+1}^2)$ and a pair of sets $A^0, A^1$ as in the definition of the event $ \S(x, L_{k+1})$.
  In other words, assume that
  \begin{itemize}
  \item[(a)] $A^0$ and $A^1$ are connected and contained in $ B(x, 3 L_{k+1})$,
  \item[(b)] their diameters are at least $ L_{k+1}/100$, and
  \item[(c)] no open path inside $ B(x, 3L_{k+1}^2)$ connects $A^0$ and $A^1$.
  \end{itemize}
  Then, for every path $\sigma$ in $ B(x, L_{k+1}^2)$ connecting $A^0$ to $A^1$ there exists $y \in K$ such that
  \begin{itemize}
  \item[(i)] $\sigma$ intersects $ B(y, L_k)$ and
  \item[(ii)] the event $\S(y,L_k)$ holds.
  \end{itemize}
  See also Figure~\ref{f:six_balls}.
\end{Lemma}

\begin{proof}
  The proof of this lemma essentially follows the steps of the proof of Lemma~3.2 in \cite{2014arXiv1409.5923T}.
  Therefore, we will not repeat the entire argument here.
  Instead, we just indicate what substitutions should be done to make that proof match exactly the context of the present article.
  First, replace each occurrence of $B(y, j L_k/6)$, for $j = 1, 2, 3$, by $B(y, j L_k)$.
  Then replace the balls $B(y, j L_k/6)$, for $j = 4, 5$ and $6$ with $B(y, (j-3) L_k^2)$.
\end{proof}

The above lemma will allow us to reduce Condition~\ref{cond:joao} to the following simpler condition, which only concerns the geometry of $G$, not the realization of the percolation process.

\nc{c:cond_2}
\begin{Condition}
  \label{cond:avoid_balls}
  We say that a graph $G$ satisfies Condition~\ref{cond:avoid_balls} if for any $J \geq 1$ there exists a constant $\uc{c:cond_2} = \uc{c:cond_2}(G, J, \gamma)$ for which the following holds.
  Given $x \in V$, a scale $k \geq \uc{c:cond_2}$, connected sets $A^0, A^1 \subseteq B(x, 3 L_{k+1})$ with diameters at least $L_{k+1}/100$ and any collection $y_1 \dots, y_{J-1} \in B(x, 2 L_{k+1}^2)$, there exists a path $\sigma$ contained in $B(x, L_{k + 1}^2)$, connecting $A^0$ with $A^1$ while avoiding the set of balls $\mcup\nolimits_{j \leq J-1} B(y_j, 12 L_k^2)$.
\end{Condition}

\begin{Lemma}
  \label{lemma:joao_2}
  Condition~\ref{cond:avoid_balls} implies Condition~\ref{cond:joao}.
\end{Lemma}

The proof of this lemma will be a consequence of Lemma~\ref{lemma:new_lemma_3.2}.

\begin{proof}
Suppose that $k\geq \uc{c:cond_2}$.
In order to establish Condition~\ref{cond:joao}, we first fix $x \in V$, a set $K \subseteq B(x, 2L_{k+1}^2)$ such that $B(K, L_k)$ covers $B(x, L_{k+1}^2)$ and assume that the event $\S(x, L_{k+1})$ holds.
  We now need to show that the events $\S(y_j, L_k)$ occur for several points $y_1, \dots, y_J$, which will be done using induction in $j = 1, \dots, J$.

  The occurrence of $\S(x, L_{k+1})$ implies the existence of sets $A^0$ and $A^1$ in $B(x, 3L_{k+1})$ as in \eqref{e:separation_event}.
  To start the induction, we use the fact that $B(x, 3L_{k+1})$ is connected to obtain a path between $A^0$ and $A^1$ and employing Lemma~\ref{lemma:new_lemma_3.2} we obtain a point $y_1 \in K$ satisfying $\S(y_1, L_k)$.
  Then, supposing that we have already found a sequence $y_1, \dots, y_{J'} \in K$ for $J' < J$ as above, we use Condition~\ref{cond:avoid_balls} to obtain a path from $A^0$ to $A^1$ that avoids $\bigcup_{j \leq J'} B(y_j, 12 L_k^2)$.
  Therefore we can use Lemma~\ref{lemma:new_lemma_3.2} again in order to obtain a new vertex $y_{J' + 1} \in K$ within distance at least $9L_k^2$ from all the previous $y_1, \dots, y_{J'}$ and for which $\mathcal{S}(y_{J' + 1}, L_k)$ holds.
  We can now continue inductively until we get Condition~\ref{cond:joao}.
\end{proof}

\subsection{Finding disjoint paths}\label{ss:disjoint_paths}

The next lemma uses the Max-Flow-Min-Cut Theorem to show that we can find several disjoint paths connecting a large set $A \subset B(x, 3 L_{k+1})$ to the (internal) boundary of the ball $B(x, 2 L_{k+1}^2)$.
By possibly trimming some of these paths, we are able to find one that avoids several balls in the previous scale.

This lemma carries some similarities with Condition~\ref{cond:avoid_balls}, however the path that one obtains is not connecting $A^0$ to $A^1$, but rather $A^0$ to far away.
This difference is at the heart of the distinction between the isoperimetric condition \eqref{eq:isoperimetric} and the \emph{local isoperimetric inequality} of \cite{2014arXiv1409.5923T}.

\nc{c:N''_large}
\begin{Lemma} \label{lemma:N''}
Suppose that a given graph $G = (V, E)$ satisfies \eqref{eq:isoperimetric} and $\mathcal{V}(c_u, d_u)$ and let $v \geq 1$ be fixed.
Then there is a constant $\uc{c:N''_large} := \uc{c:N''_large}(J, \gamma, v, c_i, d_i, c_u, d_u, \uc{c:rough_trans})$ such that the following holds.
Fixed any $x \in V$, $k \geq \uc{c:N''_large} $, any collection $z_1, \dots, z_n \in B(x, L_{k+1}^2)$ with $n \leq J \log^2(L_k)$ and any set $A \subseteq B(x, 3\uc{c:rough_trans} L_{k+1})$ with $|A| \geq L_{k+1}/(100 v)$, we have $B(x, 3\uc{c:rough_trans} L_{k+1})\subset B(x, 2 L_{k+1}^2-1)$ and there is a path from $A$ to $\partial_i B(x, 2 L_{k+1}^2)$ that does not touch the union $\bigcup_{i\leq n}B(z_i, 20 \uc{c:rough_trans} L_k^2)$.
\end{Lemma}
For every finite set $A$ define
\begin{equation}\label{eq:def_N}
N(A) := c_i |A|^{\frac{d_i-1}{d_i}},
\end{equation}
with $c_i > 0$.
Note the resemblance with \eqref{eq:isoperimetric}.

\begin{proof}
We start by showing that when $k$ is large enough, there are at least $N(A)$ edge-disjoint paths connecting $A$ to $\partial_i B(x, 2 L_{k+1}^2)$.
In fact, suppose by contradiction that this is not verified.
Then, by the Max-Flow Min-Cut Theorem, there exists a set of edges $C_A$ inside the ball $B(x, 2 L_{k+1}^2)$ which disconnects $A$ from $\partial_i B(x, 2 L_{k+1}^2)$ and such that $|C_A| < N(A)$.
Then we have
\[
|C_A |  <  N(A) = c_i|A|^{\frac{d_i - 1}{d_i}}.
\]
But then, this implies that there is a finite set $\tilde A$ (containing $A$) of points that can be reached from $A$ without using edges in $C_A$ that has to satisfy
\[
|\partial \tilde A| \leq |C_A | < c_i|A|^{\frac{d_i - 1}{d_i}} \leq c_i |\tilde A|^{\frac{d_i - 1}{d_i}},
\]
contradicting condition \eqref{eq:isoperimetric} and hence proving the first step.

We now use this fact in order to find a path that satisfies the statement of the lemma.
In fact, if we denote by $ \mathcal{M}_k(A)$ the maximal number of disjoint paths from $A$ to $\partial_i B(x, 2 L_{k+1}^2)$, using the above we have
\begin{equation}
\begin{split}
\mathcal{M}_k(A) & \geq  N(A) = c_i|A|^{\frac{d_i - 1}{d_i}}\\
& \begin{array}{e}
    & \geq & c_i\left (\frac{L_{k+1}}{100 v}\right )^{\frac{d_i - 1}{d_i}}\geq c L_k^{\gamma\frac{d_i - 1}{d_i}}\\
    & \overset{\eqref{eq:condition_gamma}, k \text{ large}}{>} & 3J\log^2(L_k)c_u(20 \uc{c:rough_trans} L_k^2)^{d_u}\\
    & \stackrel{\eqref{eq:volume_upper_bound}}{\geq } & \sup_{j} \{ 3 J \log^2(L_k) |B(y_j,20 \uc{c:rough_trans} L_k^2)|\}.
  \end{array}
\end{split}
\end{equation}
This bound shows that if we remove all those paths connecting $A$ to $\partial_i B(x, 2 L_{k+1}^2)$ which happen to intersect $\bigcup_{i\leq n}B(z_i, 20 \uc{c:rough_trans} L_k^2)$, we are still left with several paths.
\end{proof}

\subsection{Embedding a tree into \texorpdfstring{$G$}{G}}
\label{ss:embed}

In this section we will assume that Condition~\ref{cond:avoid_balls} fails, since we have already proved Theorem~\ref{thm:p_c<1_dependent} assuming Condition~\ref{cond:joao}, which follows from Condition~\ref{cond:avoid_balls}.

Negating Condition~\ref{cond:avoid_balls} is equivalent to saying that there exist
\begin{gather}
\label{e:bad_J}
\text{some number $J \geq 1$,}\\
\text{a sequence of points $x_l \in V$,}\\
\text{a diverging sequence $k_l \to \infty$ of scales,}\\
\text{connected sets $A^0_l, A^1_l \subseteq B(x_l, 3L_{k_l + 1})$ with $d(A^0_l, A^1_l) > 1$, $\diam(A^i_l) \geq L_{k_l + 1}/100$}\\
\label{eq:evil_ys}
\text{and for each $l \geq 1$ a collection $y^l_1, \dots, y^l_{J-1} \in B(x_l, 2 L_{k_l+1}^2)$}
\end{gather}
such that
\begin{display}
  \label{e:Hl_splits_As}
  every path in $B(x_l, L_{k_l+1}^2)$ connecting $A^0_l$ to $A^1_l$\\
  touches the set $\smash{\mcup_{j \leq J-1}} B(y^l_j, 12L_{k_l}^2)$.
\end{display}
We are now in position to start embedding a binary tree inside $G$, which will ultimately lead to a contradiction on the polynomial growth that we assumed on $G$.
The nodes of this tree will simply be vertices of $G$, however two adjacent vertices in the tree will not be mapped to neighbors in $G$.
Instead, they will be mapped into reasonably far apart points as we describe in detail soon.

The nodes of our binary tree are indexed by words in the alphabet $\{0,1\}$.
Let $\Gamma$ denote the set of words in this alphabet.
For every such a word $\omega \in \Gamma$, we denote by $|\omega|$ its length and by $\omega' \omega$ the word obtaining by appending $\omega$ to the right of $\omega'$.
In this case, we say that $\omega'$ is a prefix of $\omega' \omega$.
This prefix is said to be proper if $\omega$ is non-empty.

We denote the bad set
\begin{equation}
  H_l = \mcup_{j \leq J-1} B(y_j^l, 20 \uc{c:rough_trans} L_{k_l}^2).
\end{equation}
Note that the balls used to define $H_l$ have radius $20 \uc{c:rough_trans} L_{k_l}^2$, which is larger than the ones appearing in \eqref{e:Hl_splits_As}.
This difference will be important later once we start playing with rough isomorphisms.

\begin{Remark}
In the next lemma, given some $l \geq 1$ and any word $\omega \in \Gamma$ such that $|\omega| \leq \log^2(L_{k_l})$, we will construct a $\uc{c:rough_trans}$-rough isometry $\phi^l_\omega$ of $G$.
Given such a map, we can define
\begin{equation}
  \label{e:build_tree}
  \begin{array}{c}
    x_l(\omega) := \phi^l_\omega(x_l),\\
    y_j^l(\omega) := \phi^l_\omega(y_j^l),\\
    A^i_l(\omega) := \phi^l_\omega(A^i_l), \text{ $i = 0, 1$},\\
    B_l(\omega) := B\bigl (x_l(\omega), 3\uc{c:rough_trans} L_{k_l + 1}\bigr ) \text{ and}\\
    H_l(\omega) := 
    \bigcup_{j \leq J-1} B(y_j^l(\omega), 20 \uc{c:rough_trans}L_{k_l}^2).
  \end{array}
\end{equation}
Note that $\phi^l_\varnothing$ will be the identity map on $G$.

Therefore, we can think of $x_l$, $A^0_l$ and $A^1_l$ as $x_l(\varnothing)$, $A^0_l(\varnothing)$ and $A^1_l(\varnothing)$ respectively.
In the same way, we have that $y_j^{l} = y_j^{l}(\varnothing)$ for all $j$ and $l$ as above.
\end{Remark}

The next lemma constructs an embedding of a binary tree into $G$ satisfying a list of requirements.
Later we will use this together with \eqref{e:Hl_splits_As} to show that all leafs of the constructed tree have to be disjoint, contradicting the polynomial growth of the graph $G$, see Lemma~\ref{lemma:loops}.

\nc{c:build_tree}
\begin{Lemma}\label{lemma:induction}
Let $G$ be a $\uc{c:rough_trans}$-roughly transitive graph satisfying \eqref{eq:isoperimetric} and $\mathcal{V}(c_u, d_u)$ and assume the existence of a collection $(J, (x_l), (k_l), (A^0_l), (A^1_l), y_j^l)_{l \geq 1}$ as in \eqref{e:bad_J}--\eqref{eq:evil_ys}.
Then, there exists $\uc{c:build_tree} = \uc{c:build_tree}(J, \gamma, \uc{c:rough_trans}, c_i, d_i, c_u, d_u, \uc{c:N''_large})$ such that for $k_l \geq \uc{c:build_tree}$ and for each $\omega$ such that $1 \leq |\omega| \leq \log^2(L_{k_l})$, we can construct
\begin{enumerate}
\item a $\uc{c:rough_trans}$-rough isometry $(\phi^l_\omega)$ and
\item a path $\gamma_\omega$,
\end{enumerate}
in such a way that the following holds.
For $\omega$ such that $|\omega| \leq \log^2(L_{k_l}) - 1$,
\begin{gather}
  \label{e:xs_not_too_far}
  \text{if $\omega = \omega' i$, for $i = 0, 1$, then $\gamma_{\omega'} \subseteq B(x_l(\omega), L_{k_l + 1}^{3/2})$,}\\
  \label{e:B_H_disj}
  \text{for any $\omega'$ proper prefix of $\omega$, $B_l(\omega)$ is disjoint from $H_l(\omega')$,}\\
  \label{e:gamma_connects}
  \text{if $\omega = \omega'i$, with $i = 0,1$, then $\gamma_\omega$ connects $A^i_l(\omega')$ to $x_l(\omega)$ and}\\
  \label{e:gamma_H_disj}
  \text{if $\omega'$ is a proper prefix of $\omega$, then the path $\gamma_{\omega}$ is disjoint from $H_l(\omega')$}.
\end{gather}
See Figure~\ref{f:crab_party} for an illustration of the above.
\end{Lemma}

\begin{proof}
We start by choosing the constant $\uc{c:build_tree}(J, \gamma, \uc{c:rough_trans}, c_i, d_i, c_u, d_u, \uc{c:N''_large}) \geq \uc{c:N''_large}$ in such a way that for $k_l \geq \uc{c:build_tree}$,
  \begin{gather}
    \label{e:def_c_tree}
    \log^2(L_{k_l}) > J\\
    \label{e:few_defects}
    2 (J - 1)(3\uc{c:rough_trans} L_{k_l + 1} + 20 \uc{c:rough_trans} L_{k_l}^2) \log^2 L_{k_l} < L_{k_l + 1}^{3/2} - 3\uc{c:rough_trans}L_{k_l + 1}-1
  \end{gather}
  which can be done by our choice of the scales $L_k$ in \eqref{eq:condition_gamma}.
  Since we are assuming \eqref{eq:isoperimetric} and that $\uc{c:build_tree} \geq \uc{c:N''_large}$, the conclusion of Lemma~\ref{lemma:N''} is at our disposal (at each scale $k_l\geq \uc{c:build_tree}$).

  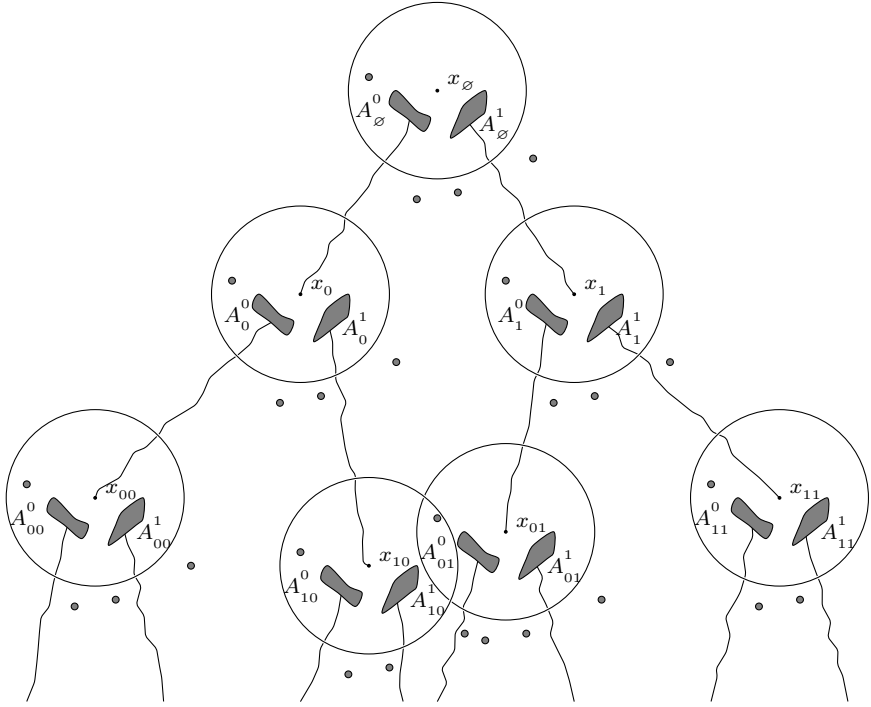
\begin{figure}
    \centering
    \begin{tikzpicture}[scale=.9]
      \draw[decorate, rounded corners=1pt, decoration={random steps,segment length=6pt,amplitude=2pt}] (0 + \htrone, 6 - \htrtwo) -- (2, 3);
      \draw[decorate, rounded corners=1pt, decoration={random steps,segment length=6pt,amplitude=2pt}] (0 - \htrone, 6 - \htrtwo) -- (-2, 3);
      \draw[decorate, rounded corners=1pt, decoration={random steps,segment length=6pt,amplitude=2pt}] (2 + \htrone, 3 - \htrtwo) -- (5, 0);
      \draw[decorate, rounded corners=1pt, decoration={random steps,segment length=6pt,amplitude=2pt}] (2 - \htrone, 3 - \htrtwo) -- (1, -.5);
      \draw[decorate, rounded corners=1pt, decoration={random steps,segment length=6pt,amplitude=2pt}] (-2 + \htrone, 3 - \htrtwo) -- (-1, -1);
      \draw[decorate, rounded corners=1pt, decoration={random steps,segment length=6pt,amplitude=2pt}] (-2 - \htrone, 3 - \htrtwo) -- (-5, 0);
      \draw[decorate, rounded corners=1pt, decoration={random steps,segment length=6pt,amplitude=2pt}] (1 + \htrone, -.5 - \htrtwo) -- (2, -3);
      \draw[decorate, rounded corners=1pt, decoration={random steps,segment length=6pt,amplitude=2pt}] (1 - \htrone, -.5 - \htrtwo) -- (0, -3);
      \draw[decorate, rounded corners=1pt, decoration={random steps,segment length=6pt,amplitude=2pt}] (-1 + \htrone, -1 - \htrtwo) -- (-.5, -3);
      \draw[decorate, rounded corners=1pt, decoration={random steps,segment length=6pt,amplitude=2pt}] (-1 - \htrone, -1 - \htrtwo) -- (-2, -3);
      \draw[decorate, rounded corners=1pt, decoration={random steps,segment length=6pt,amplitude=2pt}] (-5 + \htrone, - \htrtwo) -- (-4, -3);
      \draw[decorate, rounded corners=1pt, decoration={random steps,segment length=6pt,amplitude=2pt}] (-5 - \htrone, - \htrtwo) -- (-6, -3);
      \draw[decorate, rounded corners=1pt, decoration={random steps,segment length=6pt,amplitude=2pt}] (5 + \htrone, - \htrtwo) -- (6, -3);
      \draw[decorate, rounded corners=1pt, decoration={random steps,segment length=6pt,amplitude=2pt}] (5 - \htrone, - \htrtwo) -- (4, -3);
      \crab{ 0}{  6}{ \varnothing}
      \crab{ 2}{  3}{ 1}
      \crab{-2}{  3}{ 0}
      \crab{ 1}{-.5}{01}
      \crab{-1}{ -1}{10}
      \crab{-5}{  0}{00}
      \crab{ 5}{  0}{11}
    \end{tikzpicture}
    \caption{An illustration of the isomorphisms $\phi^l_\omega$ and the paths $\gamma_\omega$ defined in Lemma~\ref{lemma:induction}.
      Note that the points $x_l(\omega)$ and the sets $A^i_l(\omega)$ are images under $\phi^l_\omega$.
      The small gray circles correspond to the sets $H_l(\omega)$.}
    \label{f:crab_party}
  \end{figure}

  In order to construct the maps $\phi^l_\omega$, we follow an induction argument on the length of the word $\omega$.
  The only word of length zero is $\varnothing$ and we have already defined $\phi_\varnothing^l$ as the identity map.
  Assume that for $n \leq  \log^2(L_{k_l})-1$ we have already constructed the maps $\big(\phi^l_\omega\big)_{|\omega| \leq n}$ and paths $(\gamma_\omega)_{1 \leq |\omega| \leq n}$, satisfying \eqref{e:xs_not_too_far}--\eqref{e:gamma_H_disj}.
  Then, given any word $\omega$ with $|\omega| = n$, our task is now to define $\phi^l_{\omega 0}$ and $\phi^l_{\omega 1}$ with help of Lemma~\ref{lemma:N''}.

  To apply Lemma~\ref{lemma:N''}, we need to choose the points $z_1, \dots, z_m$ to be avoided, which roughly speaking will correspond to the points $\{y^l_j(\omega')$, for each $\omega'$ prefix of $\omega\}$.
  More precisely, we denote by $\omega_k$ the unique prefix of $\omega$ with $|\omega_k| = k$ and set
  \begin{gather}
    z_{k(J-1) + j - 1} = y^l_j (\omega_k), \text{ with $k = 0, \dots, |\omega|$ and $j = 1, \dots, J - 1$.}
  \end{gather}

  Recall that $|\omega| \leq \log^2(L_{k_l})-1$, so that the number of $z_i$'s is no larger than $J \log^2(L_{k_l})$.
  Using \eqref{e:large_image}, we conclude that
  \begin{equation*}
    |A_l^0(\omega)| \geq \frac{|A_l^0|}{c_u (\uc{c:rough_trans})^{d_u}} \geq \frac{L_{k_l + 1}}{100 v},
  \end{equation*}
  where, as we have mentioned, $v = c_u (\uc{c:rough_trans})^{d_u}$ (cf.\ Remark \ref{remark:note_v}).
  The same also being true for $A_l^1(\omega)$.
  We are now in position to apply Lemma~\ref{lemma:N''}, which provides us with paths $\gamma_{\omega 0}$ from $A^0_l(\omega)$ to $\partial_i B(x_l(\omega), 2\uc{c:rough_trans} L_{k_l + 1}^2)$ and $\gamma_{\omega 1}$ from $A^1_l(\omega)$ to $\partial_i B(x_l(\omega), 2\uc{c:rough_trans} L_{k_l + 1}^2)$ that satisfy \eqref{e:gamma_H_disj}.
  More precisely, these paths are such that
  \begin{display}
    \label{e:sneak}
    $\gamma_{\omega 0}$ and $\gamma_{\omega 1}$ do not touch the union of the balls $B(z_i, 20 \uc{c:rough_trans} L_{k_l}^2)$, for any $i$.
  \end{display}
  These paths will give rise to the two children of $\omega$ ($\omega 0$ and $\omega 1$).
  Recall that these paths go quite far, reaching distance $2\uc{c:rough_trans} L_{k_l + 1}^2$ from $x_l(\omega)$, however we are going to truncate these paths earlier in such a way that \eqref{e:xs_not_too_far} holds and moreover
  \begin{display}
    \label{e:stop_far}
    the end points of the paths $\gamma_{\omega 0}$ and $\gamma_{\omega 1}$ lie within distance at least $3 \uc{c:rough_trans} L_{k_l + 1}$ from $H_l(\omega')$ for any $\omega'$ prefix of $\omega$.
  \end{display}
  Before proving the above, let us briefly see why this would finish the proof of the lemma.
  We call these end-points $x_l(\omega 0)$ and $x_l(\omega 1)$ respectively and using the rough transitivity of the graph $G$ we can find two $\uc{c:rough_trans}$-rough isometries, satisfying $\phi^l_{\omega 0} (x_l(\varnothing)) = x_l(\omega 0)$ and $\phi^l_{\omega 1}(x_l(\varnothing)) = x_l(\omega 1)$.
  We can now define $A_l^i$, $B_l$ and $H_l$ as in \eqref{e:build_tree}, obtaining another layer of the tree.
  The fact that these satisfy \eqref{e:B_H_disj}--\eqref{e:gamma_H_disj} is a consequence of their construction, \eqref{e:sneak} and \eqref{e:stop_far}.

  We still need to prove that we can stop the paths $\gamma_{\omega 0}$ and $\gamma_{\omega 1}$ in such a way that they satisfy \eqref{e:xs_not_too_far} and \eqref{e:stop_far}.
  First observe that a point $x$ being within distance at least $3\uc{c:rough_trans}L_{k_l + 1}$ from the sets $H_l(\omega')$ (for $\omega'$ prefix of $\omega$) is equivalent to $x$ being within distance $3\uc{c:rough_trans}L_{k_l + 1} + 20 \uc{c:rough_trans} L_{k_l}^2$ from the collection of points $K = \{y^l_j(\omega'); \text{ $\omega'$ prefix of $\omega$ and $j \leq J - 1$} \}$.
Hence we stop these paths as soon as they reach distance $\lfloor L_{k_l + 1}^{3/2}\rfloor $ from $x_l(\omega)$ (recall that they reach $\partial_i B(x_l(\omega), 2\uc{c:rough_trans}L_{k_l + 1}^2)$), therefore $\gamma_{\omega 0}$ and $\gamma_{\omega 1}$ will automatically satisfy \eqref{e:xs_not_too_far}.

  Even after this truncation, the ranges of these paths still have diameter at least $L_{k_l + 1}^{3/2} - 3\uc{c:rough_trans}L_{k_l + 1}$.
  Therefore, by \eqref{e:few_defects} they cannot be covered by $(J - 1) \log^2 L_k$ balls of radius $3\uc{c:rough_trans}L_{k_l + 1} + 20 \uc{c:rough_trans} L_{k_l}^2$.
  This proves that we can stop the paths $\gamma_{\omega 0}$ and $\gamma_{\omega 1}$ in a way that their endpoints satisfy \eqref{e:stop_far}, finishing the proof of the lemma.
\end{proof}

In order to conclude the proof of the \nameref{lemma:joao} we will show that under the current assumptions all the points $(x_l(\omega))_{|\omega| = \lfloor \log^2(L_k) \rfloor}$ are disjoint, contradicting the polynomial growth that we have assumed on the graph $G$.

\nc{c:loops}
\begin{Lemma}\label{lemma:loops}
  There exists a constant $\uc{c:loops} = \uc{c:loops}(\gamma, \uc{c:rough_trans})$ such that for all $k_l \geq \uc{c:loops}$, if $k \geq \uc{c:build_tree}$, then we have the following.
  Let $n_l = \lfloor \log^2(L_{k_l}) \rfloor$ and fix a construction of the $\uc{c:rough_trans}$-rough isometries $\phi^l_\omega$, for $|\omega| \leq n_l$ as in Lemma~\ref{lemma:induction}.
  Then for all pair of words $\omega$, $\omega'$ such that $|\omega| = |\omega'| = n_l$, the points $x_l(\omega)$ and $x_l(\omega')$ are distinct.
\end{Lemma}

\begin{proof}
  We first fix $\uc{c:loops}$ large enough so that for $k \geq \uc{c:loops}$, one has
  \begin{equation}
    \label{e:large_c_loops}
    8 \uc{c:rough_trans} L_k^2 > 4 \uc{c:rough_trans}^2 (2\uc{c:rough_trans}+1).
  \end{equation}
  This specific choice will become clear later.

  Suppose that there are two words $\omega$ and $\omega'$, both of length $n_l$, for which
  \begin{equation}
    \label{eq:non-empty_intersection}
    x_l(\omega) = x_l(\omega')
  \end{equation}
  and let $\hat \omega$ be their closest common ancestor (in other words, $\hat{\omega}$ is the longest common prefix of $\omega$ and $\omega'$).
  Our aim is to build a path between $A^0_l$ and $A^1_l$, which is contained in $B(x_l, L_{k_l + 1}^2)$ and avoids the set $\bigcup_{j \leq J-1} B(y_j^l, 12 L_{k_l}^2)$.
  This will lead to a contradiction to \eqref{e:Hl_splits_As}, which we have obtained from negating Condition~\ref{cond:avoid_balls}.

  As a first step, we will construct a path $\sigma$ such that
  \begin{display}
    \label{e:make_handle}
    $\sigma$ is contained in $B(x_l, n_l L_{k_l + 1}^{3/2}+3\uc{c:rough_trans}n_l L_{k_l+1})$, connects $A^0_l(\hat{\omega})$ to $A^1_l(\hat{\omega})$\\
    and avoids the set $H_l(\hat{\omega})$.
  \end{display}
  Then we will use the rough inverse of $\phi^l_{\hat{\omega}}$ to ``map $\sigma$ to the desired path''.

  Before building $\sigma$, we start by constructing a path from $A^0_l(\hat{\omega})$ to $x_\omega$.
  In order to do this, we first write $\omega_0, \dots, \omega_n$ to be the sequence of prefixes of $\omega$, obtained by setting $\omega_0 = \hat{\omega}$ and adding one letter at a time until $\omega_n = \omega$.

  We start by observing that $A^0_l(\hat{\omega})$ can be connected to $x_l(\omega_1)$ by the path $\gamma_{\omega_1}$ which avoids $H_l(\hat{\omega})$ by \eqref{e:gamma_connects} and \eqref{e:gamma_H_disj}.
  Supposing by induction that we have already reached $x_l(\omega_j)$ for some $j < n$ by a path that avoids $H_l(\hat{\omega})$, we are now going to extend this path until $x_l(\omega_{j + 1})$.
  We know by \eqref{e:gamma_connects} and \eqref{e:gamma_H_disj} that if $\omega_{j+1} = \omega_j i$ ($i = 0,1$), then the path $\gamma_{\omega_{j+1}}$ connects $A^i_l(\omega_j)$ to $x_l(\omega_{j+1})$ while avoiding $H_l(\hat{\omega})$, therefore this is a good candidate for the extension we need.

  The obstacle to perform this extension comes from the fact that this path does not necessarily start at $x_l(\omega_j)$, in fact its starting point $\gamma_{\omega_{j+1}}(0)$ is somewhere in $A^i_l(\omega_j) \subseteq B(x_l(\omega_j), 3\uc{c:rough_trans} L_{k_l + 1})$, see Figure~\ref{f:crab_party}.
  But using the fact that this ball is connected and disjoint from $H_l(\hat{\omega})$ (by \eqref{e:B_H_disj}), we can connect $x_l(\omega_j)$ to $\gamma_{\omega_{j+1}}(0)$ and finally to $x_l(\omega_{j+1})$.

  Proceeding with this induction, we can construct the required path from $A^0_l(\hat{\omega})$ to $x_l(\omega)$ which avoids $H_l(\hat{\omega})$.
  We can also build a similar path from $A^1_l(\hat{\omega})$ to $x_\omega$ and by concatenating these two we have proved \eqref{e:make_handle}.

  We now use the path $\sigma$ obtained in \eqref{e:make_handle} to derive a contradiction to \eqref{e:Hl_splits_As}, finishing the proof of the lemma.
  For this, pick a $4 \uc{c:rough_trans}^2$-rough isometry $\psi$ which is a rough inverse of $\phi^l_{\hat{\omega}}$ as in \eqref{e:rough_inverse}.
  We now consider the image of the path $\sigma$ under the map $\psi$, obtaining a sequence of vertices $x_1, \dots, x_M$, for some suitable $M \geq 1$.

  This sequence does not necessarily constitute a path, however, by \eqref{e:rough_iso} we have
  \begin{display}
    \label{e:hopping_walk}
    $d(x_m, x_{m+1}) \leq 4 \uc{c:rough_trans}^2$ for every $m = 1, \dots, M-1$.
  \end{display}
  Recall that the path $\sigma$ connects $A^0_l(\hat{\omega})$ to $A^1_l(\hat{\omega})$, which are the images of $A^0_l$ and $A^1_l$ under $\phi^l_{\hat{\omega}}$.
  Therefore, the point $x_1$ (which is the image of the first point of $\sigma$) is within distance at most $ \uc{c:rough_trans}$ from $A^0_l$ (and similarly for $x_M$ and $A^1_l$).
  So we can add points $x_0 \in A^0_l$ and $x_{M+1} \in A^1_l$ to the sequence, without violating \eqref{e:hopping_walk}.

  We now use \eqref{e:hopping_walk} and the above property of $x_0$ and $x_{M + 1}$ to turn the sequence $(x_m)_{m=0}^{M+1}$ into a path by connecting $x_m$ to $x_{m+1}$, one by one, while using no more than $4 \uc{c:rough_trans}^2$ intermediate points to join each pair.
  This gives rise to a path $\sigma'$ for which we need to verify:
  \begin{enumerate} [\quad a)]
  \item $\sigma'$ connects $A^0_l$ to $A^1_l$,
  \item $\sigma'$ is contained in $B\bigl (x_l, \uc{c:rough_trans} + 4\uc{c:rough_trans} n_l^2 (L^{3/2}_{k_l + 1}+3\uc{c:rough_trans}L_{k_l+1})\bigr ) \subseteq B(x_l, L_{k_l + 1}^2)$,
  \item $\sigma'$ does not intersect the set $\bigcup_{j \leq J-1} B(y_j^l, 12 L_{k_l}^2)$.
  \end{enumerate}
  In fact, a) is a consequence of the construction of the path.
  The statement b) follows since $\psi$ is a $4\uc{c:rough_trans}^2$-rough isometry.
  Finally, to show c), we fix $y_j^l$ and $x \in B(y_j^l, 12 L_{k_l}^2)$ and, observing that
  \begin{equation}
    \label{e:phi_hat_omega_x_close}
    \phi_{\hat{\omega}}^l (x) \in B(y_j^l(\hat \omega), 12 \uc{c:rough_trans} L_{k_l}^2)
  \end{equation}
  we estimate
  \begin{equation*}
    \begin{split}
      d(\sigma', x) & \geq d \big( \sigma', \psi(\phi^l_{\hat \omega}(x)) \big) - d\big(\psi(\phi^l_{\hat \omega}(x)), x \big) \geq d \big( \psi(\sigma), \psi(\phi^l_{\hat \omega}(x)) \big) - \uc{c:rough_trans} - \uc{c:rough_trans}\\
      & \geq \frac{1}{4 \uc{c:rough_trans}^2} d(\sigma, \phi^l_{\hat \omega}(x)) - 1 - 2\uc{c:rough_trans} \overset{\eqref{e:phi_hat_omega_x_close}, \eqref{e:make_handle}}\geq \frac{1}{4 \uc{c:rough_trans}^2} (20 - 12) \uc{c:rough_trans} L_{k_l}^2 - 2\uc{c:rough_trans} - 1 \overset{\eqref{e:large_c_loops}}> 0.
    \end{split}
  \end{equation*}
  This finishes the proof that $\sigma'$ indeed contradicts \eqref{e:Hl_splits_As}, yielding the lemma.
\end{proof}

It is now very easy to finish the proof of the \nameref{lemma:joao}.

\begin{proof}[Proof of the \nameref{lemma:joao}]
  Supposing that $G$ does not satisfy Condition~\ref{cond:joao}, we know by Lemma~\ref{lemma:joao_2} that it does not satisfy Condition~\ref{cond:avoid_balls} either.
  This provides us with a sequence $(J, (k_l), (A^0_l), (A^1_l), (y^l_j))$ satisfying \eqref{e:bad_J}--\eqref{eq:evil_ys}.

Now consider all $l\geq 1$ such that $k_l\geq \max\{\uc{c:loops} , \uc{c:build_tree}\} $.
  Employing Lemma~\ref{lemma:induction}, we can construct (for each such $l $) the rough isometry $\phi^l_{\omega}$, for $|\omega| \leq n_l := \lfloor \log^2(L_{k_l}) \rfloor$ satisfying \eqref{e:xs_not_too_far}--~\eqref{e:gamma_H_disj}.

  Lemma~\ref{lemma:loops} now claims that the points $(x_l(\omega))_{|\omega| \leq n_l}$ obtained in the above construction are disjoint.
  However, there are $2^{n_l}$ such points and by \eqref{e:xs_not_too_far} they are all contained in the ball $B(o, n_l L^{3/2}_{k_l + 1})$.
  This contradicts the polynomial growth of $G$ assumed in \eqref{eq:volume_upper_bound}, finishing the proof of the \nameref{lemma:joao}.
\end{proof}


\section{Proof of Lemma \ref{lemma:lego}}
\label{s:lego}

To conclude the proof of Theorem~\ref{thm:p_c<1_dependent} we still need to show Lemma~\ref{lemma:lego}.
The main ideas of the proof are taken from \cite[Lemma 4.1]{2014arXiv1409.5923T}, which we report here for sake of clarity.
We split the proof into several auxiliary results, in order to make it more clear.

\begin{Remark}
For convenience we will prove an equivalent statement to that of Lemma~\ref{lemma:lego}, namely that for every fixed vertex $x\in V$ and $L$ large enough, under the above conditions one has $ \P[L < \text{diam}(\mathcal{C}_x) < \infty] \leq L^{-\theta}$ (recall that $\theta$ is arbitrary).
\end{Remark}

Given $x \in V$ we fix a path $\sigma:\N \to V$ that satisfies the following properties:
\begin{itemize}
\item[(i)] $ \displaystyle d(\sigma(i),\sigma(j))=|i-j|, \textnormal{ for all }i,j\in \N$;
\item[(ii)] $\displaystyle \sigma(0):=x$.
\end{itemize}
Recalling that $d(x,y)$ denotes the graph distance between vertices $x$ and $y$.
The existence of such paths will not be discussed here, but the interested reader is referred to \cite{Watkins1986341}.
(More precisely, such a path exists whenever the graph $G$ is infinite, locally finite, simple and connected.)
Now, given $\sigma$, define the following collection of points:
\[
x_{k,i}:=\sigma (i L_k/10), \textnormal{ for }k\geq 1 \textnormal{ and }i=0,\ldots, L_{k+1}/L_k,
\]
and, for some fixed $k_0\geq 1$ define the following event:
\[
\mathcal{G}_0(x):= \bigcap_{k\geq k_0}\bigcap_{i=1}^{L_{k+1}/L_k}\S(x_{k,i},L_k)^c.
\]
The next claim shows that for any fixed vertex $x$ the event $\mathcal{G}_0(x)$ occurs with high probability.

\begin{Claim}\label{claim:(4.14)}
  If for some integer $\gamma \geq 3$ we have $p_k \leq L_k^{-\beta}$ for every $k$, where $\beta > \gamma (1 + \theta)$, then the event $\mathcal{G}_0(x)$ occurs with probability bounded from below by $ 1-c L_{k_0}^{-\gamma \theta}$, where $c>0$ is a constant depending on $\theta$.
\end{Claim}

\begin{proof}
We show that $\P\left [\bigl (\mathcal{G}_0(x)\bigr )^c\right ]\leq c L_{k_0}^{-\gamma \theta}$.
In fact,
\[
\begin{split}
\P\left [\bigl (\mathcal{G}_0(x)\bigr )^c\right ]
& = \P \Bigl ( \bigcup_{k\geq k_0}\bigcup_{i=1}^{L_{k+1}/L_k}\S(x_{k,i},L_k)\Bigr )
\leq \sum_{k\geq k_0}\sum_{i=1}^{L_{k+1}/L_k}p_k
\leq \sum_{k\geq k_0}L_k^{\gamma-1}L_k^{-\beta}\\
& \stackrel{\beta>\gamma(1+\theta)}{\leq} \sum_{k\geq k_0}L_k^{-\theta \gamma} = L_{k_0}^{-\theta \gamma} \sum_{k\geq 0}(L_0^{-\theta \gamma} )^{\gamma^{k_0}(\gamma^k-1)}.
\end{split}
\]
Now, since we assumed $\gamma\geq 3$, we have $\gamma^{k_0}>1$ and $L_0^{-\theta \gamma}<1$.
Therefore, the sum $\sum_{k\geq 0}(L_0^{-\theta \gamma} )^{\gamma^{k_0}(\gamma^k-1)}$ converges, leading to the claim.
\end{proof}

The next auxiliary result shows that, on the event $\mathcal{G}_0(x)$, we can find several \emph{open} paths which can be connected, discovering an infinite (open) connected component.

\begin{Lemma}\label{lemma:(4.16)}
  On the event $\mathcal{G}_0(x)$, there is an infinite connected component that intersects the ball $B(x, L_{k_0} / 100)$.
\end{Lemma}

\begin{proof}
  We will prove that:
  \begin{display}
    For all $k\geq k_0$ there exists an open path $\sigma_k$ starting at $B(x, 3 L_k)$, contained in $B(x, 3 L_{k + 1})$ and having diameter at least $L_{k + 1}/100$.
  \end{display}

  For all $i = 0, \ldots, L_{k + 1} / L_k - 1$, we use the fact that by hypothesis (i.e., conditioning on $\mathcal{G}_0(x)$) we have that $ \S (x_{k, i}, L_k)^c$ is realized.
  Hence, we can find an \emph{open} path $\sigma_{k,i}\subseteq B(x_{k,i},3L_k^2)$ that connects the ball $B(x_{k,i}, L_k/100) $ to $B(x_{k,i+1}, L_k/100) $.
  If necessary, we truncate the paths $\sigma_k,i$ as soon as they exit the ball of radius $3 L_k$ centered at $x_{k,i}$.

  The next step consists in joining all such \emph{open} paths $ \sigma_{k,i}$'s into one (open) connected component.
  Therefore, we first estimate the diameters of such paths:
  \[
    \textnormal{diam}(\sigma_{k,i})\geq d \bigl ( B(x_{k,i}, L_k/100), B(x_{k,i+1}, L_k/100) \bigr )\geq d(x_{k,i},x_{k,i+1})-2 L_k/100\geq L_k/50.
  \]
  The last inequality follows from our choice of $\sigma$, in fact $d(x_{k,i},x_{k,i+1})\geq L_k/10$.
  Such a bound implies that before exiting the ball $ B(x_{k, i}, 3 L_k)$, the path $\sigma_{k, i}$ has diameter at least $L_k/100$.
  At this point, since we are under the assumption that $\mathcal{G}_0$ holds, we can find again \emph{open} paths $\gamma_{k,i}$ that join $\sigma_{k,i}$ with $\sigma_{k,i+1}$ (for all $i=0,\ldots,L_{k+1}/L_k-2$) that are contained inside the ball $B(x_{k,i}, 3L^2_k)$.

  Our next step is to join the $\sigma_{k,i}$'s and the $\gamma_{k,i}$'s in order to obtain longer open paths.
  Note that the paths $\gamma_{k,i}$ are necessary to avoid any issue coming from the fact that the balls $B(x_{k,i}, L_k/100)$ are not necessarily open.

  We now join such open paths, defining $\sigma_k$ that goes through $\sigma_{k,i}$ and $\gamma_{k,i}$ alternatingly, for \emph{all} values of $i=0,\ldots, L_{k+1}/L_k-1$.
  Now, by construction, we have
  \[
    \textnormal{diam}(\sigma_k)\geq d(x_{k,0},x_{k,L_{k+1}/L_k-1})-\frac{2L_k}{100}\geq \left ( \frac{L_{k+1}}{L_k}-1\right )\left ( \frac{L_k}{10}\right )- \frac{2L_k}{100}\geq \frac{L_{k+1}}{100}.
  \]
  At this point, observe that $\sigma_k$ can, a priori, lie inside the ball $B(\sigma (0), L_{k+1}+3L_k^2)$, which would not be enough for our purposes, as we need $\sigma_k$ to be contained inside $B(\sigma (0), 3L_{k+1})$.
  But by the assumption $\gamma>2$, we have that $B(\sigma (0), L_{k+1}+3L_k^2)\subset B(\sigma (0), 3L_{k+1})$, whenever $k$ is large enough (larger than some constant dependent on the value of $\gamma$).

  Now observe that since $\mathcal{G}_0$ is realized, for all $k$ large enough, the paths $\sigma_k$ and $\sigma_{k+1}$ must be on the same (open) connected component.
  In fact, since we are assuming $\S(x_{k,0},L_{k+1})^c$, before  $\sigma_k$ and $\sigma_{k+1}$ can find ``a way out'' from the ball $ B(x_{k,0}, L_{k+1})$, they will have already gained a diameter of at least $L_{k+1}/100$.

  Now the existence of all the $\sigma_k$, for $k \geq k_0$ and the fact that they all belong to the same connected component gives the statement.
\end{proof}

The next result gives a sufficient condition that will imply  Lemma~\ref{lemma:lego}.

\begin{Claim}\label{claim:(4.13)}
  Assuming that $\mathcal{G}_0(x)$ is realized (for some fixed  $x\in V$), there exists a unique infinite (open) connected component $\mathcal{C}_\infty$.
  Moreover, denoting by $\mathcal{C}_x$ the connected component of $x $, either $\mathcal{C}_x = \mathcal{C}_\infty$, or $\textnormal{diam}(\mathcal{C}_x) \leq L_{k_0}$.
\end{Claim}

\begin{proof}
  First of all, observe that the infinite cluster has to be unique due to $\mathcal{G}_0(x)$, since the existence of two or more infinite components would imply that $S(x_{k,0},L_k)$ holds for all but finitely many $k$'s.

  Furthermore, the fact that either $\mathcal{C}_x = \mathcal{C}_\infty$ or $\textnormal{diam}(\mathcal{C}_x)\leq L_{k_0}$ is a consequence of the following observation.
  If $\textnormal{diam}(\mathcal{C}_x)> L_{k_0}$, but $\mathcal{C}_x\neq \mathcal{C}_\infty$, we would find two large separated components intersecting the ball $ B(x_{k,0}, L_{k_0})$.
  But this fact would contradict the assumption that $\mathcal{G}_0(x)$ occurs.
  Hence the statement is proven.
\end{proof}

Finally we have everything in place to prove Lemma~\ref{lemma:lego}.
\begin{proof}[Proof of Lemma~\ref{lemma:lego}]
By putting together Claims~\ref{claim:(4.14)} and \ref{claim:(4.13)} we obtain the first half of the Lemma.
%
Regarding the second part, given $\ell \geq 1$, pick $k(\ell)$ such that $L_{k(\ell)} \leq \ell < L_{k(\ell) + 1}$.
Observe also that for every value $\ell$ large enough, we have
\begin{equation}
  \begin{split}
    \P \bigl (\ell < & \textnormal{diam}(\mathcal{C}_x) < \infty \bigr ) \leq \P \bigl (L_{k(\ell)}<\textnormal{diam}(\mathcal{C}_x) <\infty \bigr )\\
    & \stackrel{\textnormal{Claims~\ref{claim:(4.14)},~\ref{claim:(4.13)}}}{\leq } \ cL_{k(\ell)}^{-\gamma \theta} \leq cL_{k(\ell) + 1}^{-\theta} \leq c' \ell^{-\theta}.
  \end{split}
\end{equation}
This concludes the proof of Lemma~\ref{lemma:lego}.
\end{proof}
\begin{Remark}
Note that we can omit the constant $c'=c'(\theta)$ by proving the result with a different value of $\theta$ and then considering $k$ large enough.
\end{Remark}

\section{Examples}\label{s:examples}

This section is devoted to giving some examples of dependent percolation processes for which our results apply.
These examples include loop soups, germ-grain models and divide and color percolation.

\subsection{Loop soups}

The model of loop soups was informally introduced by Symanzik in \cite{Sym69} and was rigorously defined in \cite{LW04} in the context of Brownian loops.
The model has been intensively studied, see for example \cite{zbMATH06093904} and \cite{zbMATH06340288}, displaying some very interesting percolation features, see \cite{2014arXiv1403.5687C}.

To properly define this model, we start by introducing a space of closed loops on $G$
\begin{equation*}
  W = \Big\{(x_0, \dots, x_{k-1}) \in V^k; k \geq 1, \text{$x_0 = x_{k-1}$ and $\{x_i, x_{i-1}\} \in E$ for all $i < k$} \Big\}.
\end{equation*}

We now fix a parameter $\kappa > 0$ and endow the countable space $W$ with the measure
\begin{equation}
  \label{e:loop_measure}
  \mu(w) = \frac{1}{k} \Big( \frac{1}{\Delta(1 + \kappa)} \Big)^k,
\end{equation}
where $k$ gives the length of the loop $w$ and $\Delta$ is the maximal degree of a vertex in $G$ (which is indeed finite under assumption \eqref{eq:volume_upper_bound}).

We define an equivalence relation on $W$, where we identify two loops (denoting this by $w \sim w'$) if they have the same path length $k$ and $w(i) = w'(i + j)$ for some $j \geq 1$, where the sum is taken on $\mathbb{Z}/(k \mathbb{Z})$.

Given the equivalence relation $\sim$, we define the space of unmarked loops $W^*$ as $W / \sim$ and define the push forward $\mu^*$ of $\mu$ under the canonical projection from $W$ to $W^*$.
The process we are interested in is a Poisson Point Process $\omega^\beta$ on $W^*$ with intensity $\beta \mu^*$, where $\beta > 0$ is a parameter controlling the amount of loops that enter the picture.

We will be interested in both the occupied and vacant set left by the loop soup, or more precisely:
$\mathcal{L}^\beta = \bigcup_{w \in \text{supp}(\omega^\beta)} \text{Range}(w)$ and $\mathcal{V}^\beta = V \setminus \mathcal{L}^\beta$.

Let us state a decoupling inequality for this model, inspired by the (2.15) of \cite{Szn09}.

\begin{Lemma}
  \label{l:decouple_loops}
  Fix $\beta, \kappa > 0$.
  Then, for $r \geq 1$, $J \geq 2$, points $x_1, x_2, \ldots,x_{J} \in V$ satisfying
  \[
  \min_{1\leq i < j\leq J}d(x_i,x_j)\geq 3r
  \]
  and for events $\mathcal{G}_1, \ldots, \mathcal{G}_{J} $ such that $ \mathcal{G}_i \in \sigma (Y_z,  z\in B(x_i, r))$ we have
  \begin{equation}\label{eq:(7.2)}
    \mathbb{P}(\mathcal{G}_1 \cap \dots \cap \mathcal{G}_J) - \mathbb{P}(\mathcal{G}_1) \cdots \mathbb{P}(\mathcal{G}_J) \leq 2 \beta J \exp\{- c(\kappa) r\}\overline{v}_G(r).
  \end{equation}
\end{Lemma}

\begin{proof}
  Let us first define the sets
  \begin{equation}
    W_i = \big\{ w \in W; \Range(w) \subseteq B(x_i, 3r/2) \big\}.
  \end{equation}
  We denote by $\omega^\beta_i$ the Poisson point process $\omega^\beta$ restricted to $W_i$, for $i = 1, \dots, J$.
  Note that the $\omega^\beta_i$'s are independent, since their supports $W_i$ sets are disjoint.

  Writing $\mathcal{G}'_i$ for the event $\mathcal{G}_i$ evaluated for the trimmed point process $\omega^\beta_i$, we can estimate
  \begin{equation}\label{eq:(7.4)}
    \begin{split}
      \big| \mathbb{P}(\mathcal{G}_1 & \cap \dots \cap \mathcal{G}_J) - \mathbb{P}(\mathcal{G}_1) \cdots \mathbb{P}(\mathcal{G}_J) \big|\\
      & \leq \big| \mathbb{P}(\mathcal{G}_1 \cap \dots \cap \mathcal{G}_J) - \mathbb{P}(\mathcal{G}'_1 \cap \dots \cap \mathcal{G}'_J) \big| + \big| \mathbb{P}(\mathcal{G}_1) \cdots \mathbb{P}(\mathcal{G}_J) - \mathbb{P}(\mathcal{G}'_1) \cdots \mathbb{P}(\mathcal{G}'_J) \big|\\
      & \leq 2 J \smash{\sup_i} \mathbb{P}[\mathcal{G}_i \Delta \mathcal{G}'_i] \leq 2 J \sup_i \mathbb{P}\Big[
      \begin{array}{c}
        \text{there is $w \in \text{supp}(\omega^\beta)$ intersecting}\\
        \text{both $B(x_i, r)$ and $B(x_i, 3r/2)^c$}
        \end{array}\Big].
    \end{split}
  \end{equation}
  In order to bound the last term in the above equation we make use of the definition of the intensity measure in \eqref{e:loop_measure}, finishing the proof of the lemma.
\end{proof}

We are now in position to state the first application of our main result.

\begin{Theorem}
  Given a $\uc{c:rough_trans}$-roughly transitive graph $G$ satisfying $\mathcal{V}(c_u, d_u)$ and $\mathcal{I}(c_i, d_i)$ for some $d_i > 1$ and fix $\kappa > 0$, define the Poisson Point Processes on $G$ as above.
  Then,
  \begin{enumerate}[\quad a)]
  \item For $\beta > 0$ small enough, almost surely the set $\mathcal{L}^\beta$ contains no infinite connected component, while $\mathcal{V}^\beta$ contains a unique one.
  \item On the other hand, if $\beta > 0$ is large enough, almost surely there exists an infinite cluster in $\mathcal{L}^\beta$, whereas $\mathcal{V}^\beta$ is composed solely of finite components.
  \end{enumerate}
\end{Theorem}

This result allows us to define two critical values corresponding to the appearance of infinite clusters in $\mathcal{L}^\beta$ and $\mathcal{V}^\beta$.

\begin{proof}
The proof of the first part of the above theorem follows directly from Theorems~\ref{thm:p_c<1_dependent} and \ref{thm:p_c>0_dependent} once we apply Lemma~\ref{l:decouple_loops} (note that the decoupling provided by Lemma~\ref{l:decouple_loops} improves as we decrease $\beta$).

Let us now turn to the proof of the second part of the above theorem.
Note first that the hypothesis $\mathcal{V}(c_u, d_u)$ implies that $G$ has uniformly bounded degrees.
The assumptions $\mathcal{V}(c_u, d_u)$ and $\mathcal{I}(c_i, d_i)$ allow us to apply Theorem~\ref{thm:p_c_Bernoulli} on $G$.
All these observations together imply that for Bernoulli independent percolation on $G$ we have $p_c(G) \in (0, 1)$.
Note that we will not make use of our results for dependent percolation for this part of the proof.

We have concluded so far that, for $p$ close enough to one,
\begin{display}
the open set resulting from Bernoulli percolation contains almost surely an infinite connected component, while the closed set consists almost surely of only finite clusters.
\end{display}
It is clear that both events above are monotone increasing, hence it suffices to show that
\begin{display}
  \label{e:Bernoulli_dominates_loops}
  for any $p \in (0, 1)$, there exists $\beta(p) > 0$ such that the law of $\mathcal{L}^\beta$ stochastically dominates a Bernoulli i.i.d~percolation with parameter $p$.
\end{display}
(Subsequently, the comparison with Bernoulli percolation will conclude the proof.)

To finish we observe that the above claim can be derived by considering solely the loops with zero length in $\mathcal{L}^\beta$, which are independent due to the Poisson character of this percolation.
This finishes the proof of the second case.
\end{proof}

\begin{Remark}
  Note that some of the above arguments are more general than for Loop Soups only, in fact most of the above should work for any \emph{Germ-Grain} model.
  These models are defined as a decorated Poisson Point process, where each point gets associated with a random object to be inserted in the graph.
  Under conditions that the random objects have sufficiently light tails (for instance, exponentially bounded), then the above proof should work equally well for such models.
\end{Remark}

\subsection{Divide and color}
The divide and color model was introduced by H\"{a}ggstr\"{o}m in \cite{haggstrom_coloring}, and it is a process that is governed by two parameters ($p, q \in [0, 1]$) and evolves in two steps.
In this section we will follow the description in \cite{BBT13}, to which the reader is referred for more details and further results.
\begin{enumerate}
\item Firstly we perform a Bernoulli percolation on the edges of $G$, i.e. each edge of the given graph is retained with probability $p$, independently of each other.
  This partitions the vertices of $G$ into clusters, corresponding to the connected components induced by open edges.
\item Secondly we color the resulting connected components either \emph{black} or \emph{white} with probability $q$ or $1-q$ respectively, independently for distinct components.
  All vertices of a component take the same color, which induces dependence in this site percolation model.
\end{enumerate}

The rest of this section is devoted to proving that the decoupling condition \eqref{e:decouple_various} holds true for this model under some conditions on the parameter $p$.
In order to do so, we need to introduce some further notation.

We start by defining a Bernoulli percolation by associating at each edge $e$ a random variable $\eta \in \{0,1\}$ that takes value $\eta(e) = 1$ with probability $p$ and $\eta(e) = 0$ with probability $1 - p$.
Given such an assignment, the vertices of $G$ can be split into clusters and we associate a random variable $\xi(\mathcal{C}) \in \{0,1\}$ to each \emph{connected component} $\mathcal{C}$ determined above.
The variables $\xi(\mathcal{C})$ are i.i.d. and satisfy $\P(\xi = \textnormal{black}) = q$ and $\P(\xi = \textnormal{white}) = 1 - q$.

Finally, we re-open all edges (essentially forgetting the variables $\eta(e)$) and we ask ourselves whether there exists an infinite cluster of black sites in the above coloring.

Let $\mu_{p,q}$ denote the measure governing the site-percolation process as described above.
Then \cite[Proposition 2.5]{haggstrom_coloring} assures that for any graph $G$ and any $p \in [0, 1]$ there exists a critical value $q_\star^G(p)\in [0,1]$ such that
\[
\mu_{p,q} (\text{there exists an infinite black }q\text{-cluster})
= \left \{
\begin{array}{ll}
=0 & \text{ if }q < q_\star^G(p),\\
>0 & \text{ if }q > q_\star^G(p).
\end{array}
\right .
\]

It is clear that if $p > p_c(G)$, or in other words if the first stage of the process can lead to an infinite cluster, then $q_\star(p) = 0$, since for every positive $q$ there is a chance that the cluster containing the origin is infinite and is painted black.
Therefore, one can focus on the subcritical and critical phases $p \leq p_c(G)$.

On the subcritical phase, there is a strong belief that the size of a typical cluster should have exponential tails.
To make this more precise, let us define the critical value for ``strong subcriticality''.
We set
\begin{equation}
 \overline{p}_\ast:= \overline{p}_\ast (G) := \sup\Bigl \{p \in [0,1]; \text{ for some $\theta > 0$, } \mathbb{P}_p[\diam(\mathcal{C}_o) \geq n] \leq \exp\{\theta n\}\Bigr \}.
\end{equation}
Note that $ \overline{p}_\ast$ is independent of the reference vertex $o$.

It is clear that $ \overline{p}_\ast$ is smaller than $p_c$ and it is commonly believed that $p_c =  \overline{p}_\ast$ for a large variety of graphs.
This equality has been proved for the $d$-dimensional lattice in \cite{aizenman1987} and \cite{zbMATH03996823} and later extended to transitive graphs in \cite{AV08} and \cite{DCopinTassion15}.

Another important observation is that for any graph $G$ with degrees bounded by $\Delta$, we have $ \overline{p}_\ast \geq 1/\Delta$, as one can easily prove by a counting path argument.

Intuitively speaking, once $p <   \overline{p}_\ast$, then the clusters are small and the dependence of the \emph{divide and color} model should be short-ranged.
This is made precise in the following proposition.

\begin{Proposition}
  Fix a graph $G$ of sub-exponential growth, and let $p <  \overline{p}_\ast(G)$.
  Then, for any $\alpha > 0$ there exists a constant $c_\alpha = c_\alpha(G, p, \alpha, J)$ for which the condition \eqref{e:decouple_various} holds for the divide and color model on $G$ for any $q \in [0,1]$.

  As a consequence, if $G$ is roughly transitive, has polynomial growth and isoperimetric dimension larger than one, then $0<q_\star ^G(p) < 1$.
\end{Proposition}

\begin{proof}
  Given $x_1, \dots, x_{J} \in V$ at mutual distance at least $ 3r$ and $r \geq 1$, we are going to construct a simple decoupling of what happens in the various regions $B(x_j, r)$.
  For this, we define $J + 1$ independent percolation measures on $G$.
  More precisely, let $(\eta_j(e))_{e \in E}$ be independent Bernoulli variables (all of them i.i.d. with parameter $p$), for $j = 0, 1, \dots, J$.
  We also define a mixed configuration $\eta_\textnormal{mix}$ which is given by
  \begin{equation}
     \eta_\textnormal{mix}(e) =
     \begin{cases}
       \eta_j(e), \qquad & \text{if $e \subset B(x_j, 3r/2)$ and }\\
       \eta_0(e), & \text{otherwise.}
     \end{cases}
  \end{equation}
  We now use the above configurations to construct $J + 1$ instances of the \emph{divide and color} model, which will be denoted by $(Y^j_x)_{x \in V}$, $j = 1, \dots, J$ and $j =$ mix.

  Obviously, they use the clusters determined by their respective edge configuration $\eta^j$ defined above.
  Moreover, we add the restriction that if a given cluster of $\eta_j$ is contained in $B(x_j, 3r/2)$ (in which case it coincides with that of $\eta_\textnormal{mix}$), then both $Y^j$ and $Y^\textnormal{mix}$ will assign the same color to this cluster during the coloring stage.

  Note that $Y^\textnormal{mix}$ has the correct law of the model, and we are now in position to prove that it satisfies \eqref{e:decouple_various}.
To this purpose, start by observing that for any sequence of fixed events $\mathcal{G}_1, \dots, \mathcal{G}_J$ chosen as in Proposition~\ref{claim:lemma4.2}, relation \eqref{e:decouple_various} is guaranteed whenever
\[
\mathbb{P} \big(\mathcal{G}_1  \cap \dots \cap \mathcal{G}_J \big) - \mathbb{P} \big( \mathcal{G}_1\big) \cdots \mathbb{P}\big( \mathcal{G}_J \big)\leq c_\alpha^J r^{-J\alpha}.
\]
Then estimate
  \begin{equation}
    \begin{split}
      \mathbb{P} \Big(\mathcal{G}_1(Y^\textnormal{mix}) & \cap \dots \cap \mathcal{G}_J(Y^\textnormal{mix}) \Big) - \mathbb{P} \Big( \mathcal{G}_1(Y^1)\Big) \cdots \mathbb{P}\Big(\mathcal{G}_J(Y^J) \Big)\\
      & \leq \mathbb{P} \big( Y^\textnormal{mix}_x \neq Y^j_x, \text{ for some $j \leq J$, $x \in B(x_j, r)$} \big)\\
      & \leq \mathbb{P} \big( \text{for some $j$, an open path in $\eta_j$ connects $B(x_j, r)$ to $B(x_j, 3r/2)$} \big)\\
      & \overset{p <  \overline{p}_\ast}\leq J \exp\{-\theta r\}\overline{v}_G(r).
    \end{split}
  \end{equation}
  This finishes the proof of the proposition by properly choosing the constant $c_\alpha$ (cf.\ also Remark~\ref{remark:c-alpha}).
\end{proof}

\subsection{Slow decay of dependence}
\label{ss:elipses}

Let us briefly comment on the decay of correlation that we have assumed on the law $\mathbb{P}$.
It has been proved in \cite{BLPS97}, Theorem~1.1 that if $G$ is an amenable Cayley graph, then for any $p < 1$ there exists some invariant percolation law $\bar{\mathbb{P}}$ on $G$ such that $\bar{\mathbb{P}}^*[Y_o = 1] > p$ but the set $\{x; Y_x = 1\}$ does not percolate.
In contrast with this statement, Theorem~\ref{thm:p_c<1_dependent} states the existence of an absolute value $p_\ast$ above which every percolation law satisfying $\mathcal{D}(\alpha, c_\alpha)$ admits a unique infinite open cluster.
This distinction is clearly a consequence of the quantitative decay of correlations that we have assumed through $\mathcal{D}(\alpha, c_\alpha)$.

A natural question at this point is about the sharpness of Theorem~\ref{thm:p_c<1_dependent}.
For instance, is it true that Theorem~\ref{thm:p_c<1_dependent} still holds true if we replace the polynomial decay assumption by some slower one?
To shed some light on this question, let us mention an example from \cite{TW10b}.
It consists of a family of dependent percolation measures $(\mathbb{P}^u)_{u > 0}$ that satisfy a polynomial decay of correlations.
However, the exponent $\alpha$ appearing in the decay is not sufficiently high, so that for all $u > 0$, there is $\mathbb{P}^u$-a.s. no percolation for $\{x; Y_x = 1\}$, despite the fact that $\mathbb{P}^u [Y_o = 1]$ converges to one as $u$ tends to zero.

More precisely, in \cite{TW10b} the authors define a Poisson process on $\R^d$ which determines a set of lines passing through the space.
The intensity of this process is given by a non-trivial Haar measure on the space of lines, which invariant under translations and rotations, unique up to scaling.

Having defined this process of lines, one removes from $\mathbb{R}^d$ the cylinders of radius one and axis centered in these lines.
The resulting set is called $\mathcal{V}$.
By varying the intensity of the Poisson process, a phase transition in the percolation of $\mathcal{V}$ occurs for all $d \geq 3$, see \cite[Theorems 4.1 and 5.1]{TW10b} and \cite{HST12}.

In our setting we look at the intersection of $\mathcal{V}$ and $\R^2$, where $\mathcal{V}\subset \R^3$.
In this case, the cylinders intersected with the plane consist of ellipses with random major axis size.
In Proposition~5.6 of \cite{TW10b}), they show that, for every intensity $u > 0$ of the Poisson process,
there is no infinite component in $\mathcal{V} \cap \mathbb{R}^2$.
On the other hand, the model satisfies a condition very similar to $\mathcal{D}(\alpha, c_\alpha)$ with $\alpha = 2$, see Lemma~3.3 of \cite{TW10b}.

\appendix
\section{Appendix (proof of Corollary~\ref{thm:p_c<1})}
\label{s:appendix}

In view of Theorem~\ref{thm:p_c_Bernoulli}, it is enough to show that $G$ satisfies an isoperimetric inequality of the form \eqref{eq:isoperimetric} for some suitable $d_i > 1$.
In particular, it suffices to show the statement for any connected set of $G$.

Recall that $d'>1$ is the lower bound on the polynomial order of the growth of the graph.
We first choose $\delta > 0$ such that
\begin{equation}
  \label{e:choose_delta_d_prime}
  (1 - \delta) \frac{1 + d'}{2} > 1
\end{equation}
and observe from \cite{CPC:1771424} that for every finite connected set $S \subseteq V$,
\begin{equation}
  \label{e:iso_and_diam}
  |\partial S| \geq \frac{|S|}{\diam(S) + 1}.
\end{equation}
We now split the proof into two cases:

{\bf Case 1 ($\diam(S) \leq |S|^{1 - \delta}$)}
This case is trivially dealt with using \eqref{e:iso_and_diam}.
In fact, it suffices to take $d_i$ small enough such that
\[
\frac{d_i-1}{d_i}<\delta.
\]
Also, $c_i$ can be chosen to be $1/2$.

{\bf Case 2 ($\diam(S) \geq |S|^{1 - \delta}$)}
In this case, we let $x, y \in S$ be two points realizing the diameter of $S$ and let $\gamma \subseteq S$ be a path connecting $x=\gamma(t_0)$ to $y$.
We now let
\begin{equation}
  \begin{split}
    t_0 & = 0,\\
    t_1 & = \inf \{t \geq 0; d(\gamma_t, \gamma_0) \geq \sqrt{\diam(S)}\},\\
    t_i & = \inf \{t \geq 0; d(\gamma_t, \{\gamma_0, \dots, \gamma_{t_{i - 1}}\}) \geq \sqrt{\diam(S)}\}.
  \end{split}
\end{equation}
We note that 
\begin{itemize}
\item[(i)] By the above definition, there are order $\left (\sqrt{\diam(S)}\right ) $ points in $\gamma$ needed to connect $x$ to $y$;
\item[(ii)] The balls $B(\gamma_{t_i}, \sqrt{\diam(S)}/3)$ are disjoint.
\end{itemize}
Thus:
\begin{equation}
  \sum_{i = 1}^{\sqrt{\diam(S)}} |B(\gamma_{t_i}, \sqrt{\diam(S)}/3)| \geq c' \sqrt{\diam(S)} \sqrt{\diam(S)}^{d'} \geq c' \diam(S)^{(1 + d')/2}.
\end{equation}

From \eqref{e:choose_delta_d_prime} and the fact that $\diam(S) \geq |S|^{1 - \delta}$, we note that
\begin{equation}
  c' \diam(S)^{(1 + d')/2} \geq 2 |S|.
\end{equation}
This means that at least half of the balls $B(\gamma_{t_i}, \sqrt{\diam(S)}/3)$ must have a point in $\partial S$.
Therefore, for a suitable constant $c>0$ we obtain:
\[
|\partial S| \geq c\sqrt{\diam(S)} \stackrel{\text{Case 2}}{\geq } c |S|^{(1 - \delta)/2 }.
\]
In this case, by taking $1<d_i\leq (1/2+\delta)^{-1}$ we obtain the isoperimetric inequality \eqref{eq:isoperimetric}, concluding the proof.


\def\cprime{$'$}

\end{document}